\documentclass[a4paper, reqno, 11pt]{amsart}

\usepackage{amsmath,amssymb,amsthm,amsfonts,mathtools,mathrsfs, mathscinet, bbm}
\usepackage[usenames,dvipsnames]{color}
\usepackage{url}
\usepackage{hyperref}
\hypersetup{colorlinks=true,citecolor=blue,linkcolor=Red,urlcolor=blue,pdfstartview=FitH}
\usepackage{ amscd, dsfont}
\usepackage{enumerate}
\usepackage{easybmat,graphics}
\usepackage{etex, tikz-cd, epic}
\usetikzlibrary{matrix,arrows,decorations.pathmorphing}
\usepackage{hyperref}
\usepackage[cal=boondox]{mathalfa}

\hypersetup{colorlinks=true,citecolor=blue,linkcolor=black}
\usepackage{etex}
\usepackage[all]{xy}
\usepackage[mathscr]{euscript}

\newtheorem{theorem}{Theorem}[section]
\newtheorem{lemma}[theorem]{Lemma}
\newtheorem{corollary}[theorem]{Corollary}
\newtheorem{proposition}[theorem]{Proposition}

\newtheorem{remark}[theorem]{Remark}

\theoremstyle{definition}

\newlength{\margins}
\usepackage[top=\margins,bottom=\margins,left=\margins,right=\margins]{geometry}

\numberwithin{equation}{section}

\def\max{\operatorname{max}}

\def\Ker{\operatorname{Ker}}

\def\Ad{\operatorname{ Ad}}

\def\Aut{\operatorname{Aut}}

\def\Inn{\operatorname{Inn}}

\def\der{\operatorname{{der}}}

\def\vol{\operatorname{ vol }}

\def\Gal{\mathop{\rm Gal}\nolimits}

\def\GL{\mathop{\rm GL}\nolimits}
\def\PGL{\mathop{\rm PGL}\nolimits}
\def\ad{{\mathop{\rm ad}\nolimits}}
\def\Nm{{\mathop{\rm Nm}\nolimits}}
\def\pr{{\mathop{\rm pr}\nolimits}}

\def\sc{{\mathop{\rm sc}\nolimits}}
\def\der{{\mathop{\rm der}\nolimits}}

\def \Del{\mathop{\rm Del}\nolimits}
\def \Tr{\mathop{\rm Tr}\nolimits}
\def \Kaz{\mathop{\rm Kaz}\nolimits}

\newcommand{\Res}{\text{Res}}
\newcommand{\Fu}{{\breve F}}

\numberwithin{equation}{section}

\numberwithin{equation}{section}

\newcommand{\cA}{{\mathcal{A}}}
\newcommand{\cB}{{\mathcal{B}}}
\newcommand{\cC}{{\mathcal{C}}}

\newcommand{\cF}{{\mathcal{F}}}

\newcommand{\cH}{{\mathcal{H}}}

\newcommand{\cK}{{\mathcal{K}}}
\newcommand{\cL}{{\mathcal{L}}}
\newcommand{\cM}{{\mathcal{M}}}

\newcommand{\cP}{\mathcal{P}}

\newcommand{\cT}{{\mathcal{T}}}

\newcommand{\bbC}{{\mathbb{C}}}

\newcommand{\bbG}{{\mathbb{G}}}

\newcommand{\bbR}{{\mathbb{R}}}

\newcommand{\bbZ}{{\mathbb{Z}}}

\newcommand{\tc}{\widetilde{c}}

\newcommand{\tp}{\widetilde{p}}

\newcommand{\tF}{{\widetilde{F}}}

\newcommand{\tL}{{\widetilde{L}}}

\newcommand{\fO}{\mathfrak{O}}

\newcommand{\fe}{\mathfrak{e}}

\newcommand{\fp}{\mathfrak{p}}

\newcommand{\brM}{{\breve{M}}}
\newcommand{\btau}{{\breve{\tau}}}
\newcommand{\bl}{{\breve{\lambda}}}
\newcommand{\bs}{{\breve{s}}}
\newcommand{\ba}{{\breve{a}}}
\newcommand{\by}{{\breve{y}}}
\newcommand{\be}{{\breve{\eta}}}
\newcommand{\bb}{{\breve{b}}}
\newcommand{\bcC}{{\breve{\cC}}}
\newcommand{\bcF}{{\breve{\cF}}}
\newcommand{\bc}{{\breve{c}}}

\newcommand{\bz}{{\breve{z}}}
\newcommand{\bmu}{{\breve{\mu}}}
\newcommand{\bP}{{\breve{P}}}
\newcommand{\bL}{{\breve{L}}}

\newcommand{\bF}{{\breve{F}}}
\newcommand{\bnu}{{\breve{\nu}}}

\newcommand{\bea}{{\be_{\ad}}}

\newcommand{\blambda}{{\breve{\lambda}}}

\newcommand{\red}{{\rm{red}}}

\begin{document}
\title{A Hecke algebra isomorphism over close local fields}
\author{Radhika Ganapathy}

\address{Department of Mathematics, Indian Institute of Science, Bengaluru.}
\email{radhikag@iisc.ac.in}
\begin{abstract} Let $G$ be a split connected reductive group over $\bbZ$. Let $F$ be a non-archimedean local field. With $K_m: = \Ker(G(\fO_F) \rightarrow G(\fO_F/\fp_F^m))$, Kazhdan proved that for a field $F'$ sufficiently close local field to $F$, the Hecke algebras $\cH(G(F),K_m)$ and $\cH(G(F'),K_m')$ are isomorphic, where $K_m'$ denotes the corresponding object over $F'$. In this article, we generalize this result to general connected reductive groups.
\end{abstract}
\maketitle

\section*{Introduction}

The goal of this article is to generalize Kazhdan's theory of studying representation theory of split reductive groups over close local fields to general connected reductive groups (see Theorem \ref{MainTheorem}). Let us briefly recall the Deligne-Kazhdan correspondence:\\
(a)  Given a local field $F'$ of characteristic $p$ and an integer $m \geq 1$, there exists a local field $F$ of characteristic 0 such that $F'$ is $m$-close to $F$, i.e., $\fO_F/\fp_F^m \cong \fO_{F'}/\fp_{F'}^m$.    \\
(b) In \cite{Del84},  Deligne proved that if  the fields $F$ and $F'$ are $m$-close, then
\[ \Gal(F_s/F)/I_F^m \cong \Gal(F_s'/F')/I_{F'}^m, \] where $I_F$ is the inertia subgroup and $I_F^m$ denotes the $m$-th higher ramification subgroup of $I_F$ with upper numbering. 
This gives a bijection
\begin{align*}
& \text{\{Cont., complex, f.d. representations of $\Gal(F_s/F)$ trivial on $I_F^m$\}}\\
& \longleftrightarrow  \text{\{Cont., complex, f.d. representations of $\Gal(F_s'/F')$ trivial on $I_{F'}^m$\}}. 
\end{align*}
Moreover, all of the above holds when $\Gal(F_s/F)$ is replaced by $W_F$, the Weil group of $F$.\\
(c)  Let $G$ be a split, connected reductive group defined over $\bbZ$. For an object $X$ associated to the field $F$, we will use the notation $X'$ to denote the corresponding object over $F'$.  In \cite{kaz86}, Kazhdan proved that  given $m \geq 1$, there exists $l \geq m$ such that if $F$ and $F'$ are $l$-close, then there is an isomorphism of Hecke algebras $\Kaz_m:\cH(G(F), K_m) \rightarrow \cH(G(F'), K_m')$, where $K_m $ is the $m$-th usual congruence subgroup of $G(\fO_F)$. 
Hence, when the fields $F$ and $F'$ are sufficiently close, we have a bijection
 \begin{align*}
 &\text{\{Irreducible  admissible $\bbC$-representations $(\pi, V)$ of $G(F)$ such that $\pi^{K_m} \neq 0$\}} \\
  &\longleftrightarrow\text{\{Irreducible admissible $\bbC$-representations  $(\pi', V')$ of $G(F')$ such that $\pi'^{K_m'} \neq 0$\}}. 
\end{align*}
These results suggest that, if one understands the representation theory of $\Gal(F_s/F)$ for all  local fields $F$ of characteristic 0, then one can use it to understand the representation theory of $\Gal(F_s'/F')$ for a local field $F'$ of characteristic $p$, and similarly, with an understanding of the representation theory of ${G}(F)$ for all local fields $F$ of characteristic 0, one can study the representation theory of ${G}(F')$, for $F'$ of
characteristic $p$. This philosophy has proved helpful in studying the local Langlands correspondence for split reductive groups in characteristic $p$ with an understanding of the local Langlands correspondence of such groups in characteristic 0 (see \cite{Bad02, Lem01,  Gan15, ABPS14, GV17}).

There are three crucial ingredients that go into the proof of the Kazhdan isomorphism for split reductive groups. 
\begin{enumerate}[(1)]
\item The Hecke algebra $\cH(G(F), K_m)$ is finitely presented.
\item The group $G(F)$ admits a Cartan decomposition, that is 
\begin{align}\label{Cartan0}
G(\fO_F) \backslash G(F) /G(\fO_F) = W(G,T) \backslash X_*(T)
\end{align}
where $T$ is a maximal $\bbZ$-split torus in $G$, $X_*(T)$ its cocharacter lattice and $W(G,T)$ the Weyl group of $T$ in $G$.
\item We have obvious isomorphisms\begin{align}\label{Kmiso0}
G(\fO_F)/K_m \cong G(\fO_F/\fp_F^m) \cong G(\fO_{F'}/\fp_{F'}^m)\cong G(\fO_{F'})/K_m'. 
\end{align}
if the fields $F$ and $F'$ are $m$-close.
\end{enumerate}
Using these, Kazhdan establishes the isomorphism of Hecke algebras as follows. First, since the torus $T$ is split,  the map $X_*(T) \rightarrow T(F), \; \lambda \rightarrow \lambda(\varpi_F)$ gives a (group-theoretic) section of the natural homomorphism $T(F) \rightarrow X_*(T)$, for any choice of uniformizer $\varpi_F$ of $F$. Hence $T(F) \cong T(\fO_F) \times X_*(T)$ as groups. Let $T_m = \Ker(T(\fO_F) \rightarrow T(\fO_F/\fp_F^m))$. For any field $F'$ that is $m$-close to $F$ and a choice of uniformizer $\varpi_{F'}$ of $F'$,  we get a section of $X_*(T) \rightarrow T(F')$ and a group isomorphism $T(F)/T_m \cong T(F')/T_m'$ such that, for each $\lambda \in X_*(T)$, $\lambda(\varpi_F) \mod T_m \rightarrow \lambda(\varpi_{F'}) \mod T_m'$ under this isomorphism. Next, the isomorphism in \eqref{Kmiso0} combined with the section of $T(F)$ (resp. $T(F')$) constructed above gives a nice set of representatives of the double cosets in $K_m\backslash G(F) /K_m$ (resp. $K_m'\backslash G(F') /K_m'$) such that the induced map $\cH(G(F), K_m) \rightarrow \cH(G(F'), K_m')$ is an isomorphism of $\bbC$-vector spaces. Kazhdan then uses (1) above to choose an $l>>m$ and shows that if the fields $F$ and $F'$ are $l$-close, this isomorphism of $\bbC$-vector spaces at level $m$ is in fact an algebra isomorphism. 

We now state what is known about these ingredients for general connected reductive groups:
\begin{enumerate}[(1')]
\item (1) is true for the Hecke algebra $\cH(G^*(F),P^*)$ where $G^*$ is a connected reductive group over $F$ and $P^*$ is a compact open subgroup of $G^*(F)$ by  \cite[Theorem 2.13 and Corollary 3.4]{Ber84}.
\item For a connected reductive group $G^*$ over $F$ and  a special maximal parahoric subgroup $K^*$ of $G^*(F)$, the Cartan decomposition analogous to \eqref{Cartan0} is known by the work of Haines - Rostami (see  \cite{HR10}). More precisely, they show that $K^*\backslash G^*(F) /K^* = W(G^*, A^*) \backslash \Omega_{M^*}$, where $A^*$ is a maximal $F$-split torus in $G^*$, $M^* = C_{G^*}(A^*)$ is a minimal Levi subgroup of $G^*$ and $\Omega_{M^*}$ is the Iwahori-Weyl group of $M^*$. 
\item We note that (3) is not obvious when the group is not split, and the analogue of \eqref{Kmiso0} has been established in \cite{Gan18}, when $G^*$ is a connected reductive group over $F$, $P^*$ is a parahoric subgroup of $G^*(F)$,  and $P_m^*$ is the $m$-th Moy-Prasad filtration subgroup of $P^*$. 
\end{enumerate}
With these ingredients in place for general $G^*$, we establish the Kazhdan isomorphism for the Hecke algebra $\cH(G^*(F), K_m^*)$ where $G^*$ is a connected reductive group over $F$, $K^*$ is a special maximal parahoric subgroup of $G^*(F)$ and $K_m^* = \Ker(\cK^*(\fO_F) \rightarrow \cK^*(\fO_F/\fp_F^m)$ where $\cK^*$ is the underlying smooth affine $\fO_F$-group scheme of $K^*$ constructed by Bruhat-Tits. 

The key difficulty that remains in carrying out the strategy of Kazhdan is as follows. For general $G^*$, in view of (2'), we need to choose a nice set of representatives $\{n_{\tau^*}\in M^*(F)\; |\;\tau^* \in \Omega_{M^*}\}$. Further, we need to show that when $F$ and $F'$ are $m$-close there is a group isomorphism  $M^*(F)/M_m^* \cong M'^*(F')/M_m'^*$ such that $n_{\tau^*} \mod M_m^* \rightarrow n_{\tau'^*} \mod M_m'^*$ under this isomorphism.  When $G^*$ is quasi-split, $M^*$ is a torus. If $G^*$ is not quasi-split, $M^*$ is not even commutative; in fact $M^*$ is an inner form of a connected reductive group of type $A$. This leads us to the question of understanding sections of the Kottwitz homomorphism $\kappa_{M^*, F}:M^*(F) \rightarrow \Omega_{M^*}$ (Note that when $G^*$ is split, $M^* =T$ is just a maximal split torus in $G^*$ and the Kottwitz homomorphism is just the natural homomorphism  $T(F) \rightarrow X_*(T)$). 

In Section \ref{KottwitzTorus}, we consider a general torus $T$ over $F$ and construct a group-theoretic section of the Kottwitz homomorphism $\kappa_{T,F}: T(F) \rightarrow \Omega_{T}$ (see Lemma \ref{tlsstable}). We then combine this with the work of Chai-Yu (see \cite{CY01}) and prove in Lemma \ref{NMTCLF} that $T(F)/T_m \cong T'(F')/T_m'$ as groups  provided the fields $F$ and $F'$ are sufficiently close. Next, note that with $M^*$ as in (2'), its adjoint group is anisotropic over $F$. In Section \ref{KottwitzM}, we construct a nice set-theoretic section $\tp$ of the Kottwitz homomorphism $\kappa_{M^*, F}: M^*(F) \rightarrow \Omega_{M^*}$, which is a group theoretic section if $M^*$ is itself adjoint; see Proposition \ref{GTS}.  To do this we exploit the fact that $M^*$ is of type $A$. Let $M^*(F)_1$ be the kernel of $\kappa_{M^*, F}$. The results of \cite{Gan18} yield that $M^*(F)_1/M_m^* \cong M'^*(F')_1/M_m'^*$ provided $F$ and $F'$ are sufficiently close. We prove that $\Omega_{M^*} \cong \Omega_{M'^*}$ and identify these groups via these isomorphisms. Note that unlike the case of tori, the Kottwitz homomorphism $\kappa_{M^*, F}$ need not admit a group-theoretic section for general $M^*$, so the exact sequence  $1 \rightarrow M^*(F)_1/M^*_m \rightarrow M^*(F)/M^*_m \rightarrow \Omega_{M^*} \rightarrow 1$ is not split. We prove in Proposition \ref{MmCLF} that the group extensions
\[
\begin{tikzcd}
 & M^*(F)/M^*_m\arrow[dr,"\phi"] \\
1\rightarrow M^*(F)_1/M^*_m \arrow[ur] \arrow[dr] &&\Omega_{M^*} \rightarrow 1\\
&M'^*(F')/M_m'^*\arrow[ur,"\phi'"]
\end{tikzcd} \] 
are equivalent provided the fields $F$ and $F'$ are sufficiently close. Further, for each $\tau \in \Omega_{M^*}$ (which we have identified with $ \Omega_{M'^*})$, we have $\tp(\tau) \mod M_m^* \rightarrow \tp'(\tau) \mod M_m'^*$ under this isomorphism.  Finally, after establishing some technical results in Section \ref{MT} that allow us to compare the various objects in (1') - (3') for sufficiently close local fields $F$ and $F'$, we follow the strategy of Kazhdan and prove in Theorem \ref{MainTheorem} that the Hecke algebras $\cH(G^*(F), K_m^*)$ and $\cH(G'^*(F'), K_m'^*)$ are isomorphic provided the fields $F$ and $F'$ are sufficiently close.

\section*{Acknowledgments}
I express my gratitude to J.K.Yu for introducing me to questions related to this article and for the insightful discussions during my graduate school years.  I am grateful to Thomas Haines for pointing out an error in an earlier version of the proof of Proposition \ref{MmCLF} and for some helpful correspondence regarding it. Some of the lemmas in Section \ref{KottwitzTorus} were originally part of \cite{RX}, but did not make it to the final version; I thank Xuhua He for allowing me to include it here. I thank Maarten Solleveld and Marie-France Vign\'eras for some helpful comments and suggestions. Finally, I thank the referee for many useful suggestions that have improved the presentation and the readability of this article.

\section{Notation and preliminaries}
\subsection{Deligne's theory} Let $F$ be a non-archimedean local field, $\fO_F$ its ring of integers, $\fp_F$ its maximal ideal, and $\varpi_F$ a uniformizer. Fix a separable closure $F_s$ of $F$ and let $\Gamma_F= \Gal(F_s/F)$. Let $\bF$ be the completion of the maximal unramified extension of $F$ contained in $F_s$ and let $\sigma \in \Aut(\bF/F)$ denote the Frobenius automorphism. 

Let $m \geq 1$.  Let $I_F$ be the inertia group of $F$ and  $I_F^m$ be its $m$-th  higher ramification subgroup with upper numbering (cf. \cite[Chapter IV]{Ser79}). Let us summarize the results of Deligne \cite{Del84} that will be used later in this work.  Deligne considered the triplet $\Tr_m(F) = (\fO_F/\fp_F^m, \fp_F/\fp_F^{m+1}, \epsilon)$, where $\epsilon$ = natural\pagebreak[2] projection of $\fp_F/\fp_F^{m+1}$ on $\fp_F/\fp_F^m$, and proved that $ \Gamma_F/I_F^m$, 
together with its upper numbering filtration, is canonically determined by $\Tr_m(F)$. Hence an isomorphism of triplets $\psi_m: \Tr_m(F) \rightarrow \Tr_m(F')$ gives rise to an isomorphism
\begin{equation}\label{Deliso}
\Gamma_F/I_F^m \xrightarrow{\Del_m} \Gamma_{F'}/I_{F'}^m
\end{equation}
that is unique up to inner automorphisms (see\cite[Equation 3.5.1]{Del84}). More precisely, given an integer $f \geq 0$, let $ext(F)^f$ be the category of finite separable extensions $E/F$  satisfying the following condition: The normal closure $E_1$ of $E$ in $F_s$ satisfies $\Gal(E_1/F)^f = 1$. Deligne proved that an isomorphism $\psi_m: \Tr_m(F) \rightarrow \Tr_m(F')$ induces an equivalence of categories
$ext(F)^m \rightarrow ext(F')^m$. Here is a partial description of the map $\Del_m$ (see \cite[Section 1.3]{Del84}).  Let $L$ be a finite totally ramified Galois extension of $F$ satisfying $I(L/F)^m = 1$ (here $I(L/F)$ is the inertia group of $L/F$). Then  $L = F(\alpha)$ where $\alpha$ is a root of an Eisenstein polynomial \[P(x) = x^n + \varpi_F\sum a_ix^i\] for $a_i \in \fO_F$.  Let $a_i' \in \fO_{F'}$ be such that $a_i \mod \fp_F^m \rightarrow a_i' \mod \fp_{F'}^m$. So $a_i'$ is well-defined mod $\fp_{F'}^m$.
Then the corresponding extension $L'/F'$ can be obtained as $L' = F'(\alpha')$ where $\alpha'$ is a root of the polynomial \[P'(x) = x^n + \varpi_{F'} \sum a_i'x^i\] where $\varpi_F \mod \fp_F^m \rightarrow \varpi_{F'} \mod \fp_{F'}^m$.  The assumption that $I(L/F)^m =1$ ensures that the extension $L'$ does not depend on the choice of $a_i'$, up to a unique isomorphism.
\subsection{Kazhdan's theory}\label{KazhdanIsomorphism}
Let us recall the results of \cite{kaz86}. Let $G$ be a split connected reductive group defined over $\bbZ$. Let $K_m = \Ker({G}(\fO_F) \rightarrow {G}(\fO_F/\fp_F^{m}))$ be the $m$-th usual congruence subgroup of $G$.
Fix a Haar measure $dg$ on $G$ with $\vol(K_m; dg) =1$.  The set $\{ \mathbbm{1}_{K_m x K_m}| x \in G(F)\}$ forms a $\mathbb{C}$-basis of the Hecke algebra $\cH(G(F), K_m)$ (of compactly supported $K_m$-biinvariant complex valued functions on $G(F)$).
Let \[X_*({T})_+=\{\lambda \in X_*({T}) \,| \,\langle \alpha, \lambda\rangle \,\geq 0 \; \forall\;\alpha \in \Phi^+\}.\] Let $\varpi_\lambda = \lambda(\varpi_F)$ for $\lambda \in X_*({T})_+$. Consider the Cartan decomposition of $G$:
\[ G(F) = \displaystyle{\coprod_{\lambda \in X_*({T})_+} {G}(\fO_F)\varpi_\lambda {G}(\fO_F)}.\]

The set $ {G }(\fO_F)\varpi_\lambda{G}(\fO_F)$ is a homogeneous space of the group $ {G }(\fO_F)\times {G}(\fO_F)$ under the action $(a,b).g = agb^{-1}$. 
The set $\{ {K_m}xK_m| x \in {G }(\fO_F)\varpi_\lambda {G}(\fO_F) \}$ is then a homogeneous space of the finite group ${G}(\fO_F/\fp_F^{m}) \times {G}(\fO_F/\fp_F^{m})$. 
 Let $\Gamma_\lambda \subset  {G}(\fO_F/\fp_F^{m}) \times {G}(\fO_F/\fp_F^{m})$ be the stabilizer of the double coset $K_m \varpi_\lambda K_m$.
 Kazhdan  observed that the obvious isomorphism 
  \begin{align}\label{Kmiso}
  G(\fO_F)/K_m \cong {G}(\fO_F/\fp_F^{m})  \xrightarrow{\cong}  {G}(\fO_{F'}/\fp_{F'}^{m})  \cong G(\fO_{F'})/K_m'
  \end{align}
    maps $\Gamma_\lambda \rightarrow \Gamma_\lambda'$, where $\Gamma_\lambda'$ is the corresponding object for $F'$.
  Let $T_\lambda \subset {G }(\fO_F)\times {G}(\fO_F) $ be a set of representatives of $\left({G}(\fO_F/\fp_F^{m})\times {G}(\fO_F/\fp_F^{m})\right) /\Gamma_\lambda$. 
Similarly define $T_{\lambda}'$. Then we have a bijection $T_\lambda \rightarrow T_\lambda'$. Kazhdan constructed 
an isomorphism of $\mathbb{C}$-vector spaces
\begin{align*} \cH(G(F), K_m) \xrightarrow{\Kaz_m} \cH(G(F'), K_m')
\end{align*}
by requiring that
\begin{align*}
\mathbbm{1}_{K_m a_i\varpi_\lambda a_j^{-1} K_m} \mapsto \mathbbm{1}_{K_m a_i'\varpi'_\lambda a_j'^{-1} K_m}
 \end{align*}
for all $\lambda \in X_*({T})_+$ and $(a_i,a_j) \in T_\lambda$, where
$(a_i',a_j')$ is the image of $(a_i, a_j)$ under the
bijection $T_{\lambda} \rightarrow T_\lambda'$. He then proved the following theorem.
\begin{theorem}[Theorem A of \cite{kaz86}] \label{Kaziso} Given $m \geq 1$, there exists $ l \geq m$ such that if $F$ and $F'$ are $l$-close, the map $\Kaz_m$ constructed above is an algebra isomorphism.
\end{theorem}
An irreducible, admissible representation $(\tau, V)$ of $G(F)$ such that $\tau^{K_m} \neq 0$ naturally becomes an $\cH(G(F), K_m)$-module. Hence, if the fields $F$ and $F'$ are sufficiently close,  $\Kaz_m$ gives a  bijection
 \begin{align}\label{Kazreptrans}
 &\text{\{Iso. classes of irr. ad. representations $(\tau, V)$ of $G(F)$ with $\tau^{K_m} \neq 0$\}}\nonumber \\
  & \longleftrightarrow\text{\{Iso. classes of irr. ad. representations  $(\tau', V')$ of $G'(F')$ with $\tau'^{K_m'} \neq 0$\}}.
\end{align}

The purpose of this article is to generalize Theorem \ref{Kaziso} to general connected, reductive groups.

\subsection{Summary of \cite{Gan18}}The main goal of \cite{Gan18} is to study the reduction of parahoric group schemes, attached to points or facets in the Bruhat-Tits building of a connected reductive group $G$ over $F$, mod $\fp_F^m$ and prove they are isomorphic for sufficiently close fields.  We will first recall the result of Chai-Yu \cite{CY01} for tori and then summarize the results of \cite{Gan18} that will be used in this work. In the process we will also introduce some notation that will be needed for the rest of the article.
\subsubsection{The case of tori by Chai-Yu}\label{CY}
Let $T$ be a torus over $F$. Then $T$ is determined by the $\Gamma_F$-module $X_*(T)$ up to a canonical isomorphism. Let $\cT$ be the identity component of the N\'eron-Raynaud model of $T$. 

 Let $m \geq 1$ be such that $T$ splits over an  at most $m$-ramfied Galois extension of $F$.  Then the action of $\Gamma_F$ on $X_*(T)$ factors through $\Gamma_F/I_F^m$. For any field $F'$ that is at least $m$-close to $F$, we obtain a torus $T'$ over $F'$  via the action of $\Gamma_{F'} \rightarrow \Gamma_{F'}/I_{F'}^m \xrightarrow[\cong]{\Del_m^{-1}}\Gamma_F/I_F^m$ on $X_*(T)$. This torus splits over an at most $m$-ramified extension of $F'$. Let $\cT'$ be the identity component of the N\'eron-Raynaud model of $T'$. 
 \begin{theorem}[Section 9 of \cite{CY01}]\label{MTCY} Given $m \geq 1$ there exists $e\geq m$ such that for any field $F'$ that is $e$-close to $F$, the group schemes $\cT \times_{\fO_F} \fO_F/\fp_F^m$ and $\cT' \times_{\fO_{F'}} \fO_{F'}/\fp_{F'}^m$ are isormorphic. In particular, 
\[\cT (\fO_F/\fp_F^m) \cong \cT (\fO_{F'}/\fp_{F'}^m)\] as groups. 
\end{theorem}
Next, we summarize the results of \cite{Gan18} that will be used later in the work. 
\subsubsection{The quasi-split case}\label{Gan18QS} Let $(R,\Delta)$ be a based root datum and let $(G_0, T_0, B_0, \{x_\alpha\}_{\alpha \in \Delta})$ be a pinned, split, connected, reductive $\bbZ$-group with based root datum $(R, \Delta)$. We know that the $F$-isomorphism classes of quasi-split groups $G$ that are $F$-forms of $G_0$ are parametrized	by the pointed cohomology set $H^1(\Gamma_F, Aut(R, \Delta))$. Let $E_{qs}(F, G_0)_m$  be the set of $F$-isomorphism classes of quasi-split groups $G$ that split (and become isomorphic to $G_0$) over an at most $m$-ramified extension of $F$. It is easy to see that this is parametrized by the cohomology set
$H^1(\Gamma_F/I_F^m, Aut(R, \Delta))$ (See \cite[Lemma 3.1]{Gan18}). Using the Deligne isomorphism, it is shown that there is a bijection $E_{qs}(F, G_0)_m \rightarrow E_{qs}(F', G_0')_m,\;\, G \rightarrow G',$ provided $F$ and $F'$ are $m$-close (See \cite[Lemma 3.3]{Gan18}).  Moreover, with the cocycles chosen compatibly, this will yield data $(G, T, B)$ over $F$ (where $T$ is a maximal $F$-torus and $B$ is an Borel subgroup of $G$ containing $T$ and defined over $F$), and correspondingly $(G', T', B')$ over $F'$, together with an isomorphism $X_*(T) \rightarrow X_*(T')$ that is $\Del_m$-equivariant (see \cite[Lemma 3.4]{Gan18}). It is a simple observation that the maximal $F$-split subtorus $S$ of $T$ is a maximal $F$-split torus in $G$ (see \cite[Lemma 4.1]{Gan18}). Let $\cA_m: \cA(S,F) \rightarrow \cA(S', F')$ be the simplicial isomorphism in \cite[Proposition 4.4 and Lemma 4.9]{Gan18}). Let $\cF$ be a facet in $\cA(S,F)$ and $\cF' = \cA_m(\cF)$. We then have the following theorem:

\begin{theorem}[Theorem 4.5 and Proposition 4.10 of \cite{Gan18}]\label{PGCLFQ}
Let $m\geq 1$. There exists $e>>m$ such that if $F$ and $F'$ are $e$-close, then  parahoric group schemes $\cP_{\cF} \times_{\fO_F} \fO_F/\fp_F^m$ and $\cP_{\cF'} \times_{\fO_{F'}} \fO_{F'}/\fp_{F'}^m$ are isormorphic. In particular, 
\[\cP_{\cF} (\fO_F/\fp_F^m) \cong \cP_{\cF'} (\fO_{F'}/\fp_{F'}^m)\] as groups. 
\end{theorem}
\subsubsection{The case of inner forms}\label{Gan18IF}
We recall that any connected reductive group is an inner form of a quasi-split group, and the isomorphism classes of inner twists of a quasi-split group $G$ over $F$ is parametrized by the cohomology set $H^1(\Aut(\Fu/F), G_{\ad}(\Fu))$.  With data $(G,T, B)$ corresponding to $(G', T', B')$ as above, it is shown in  \cite[Lemma 5.1]{Gan18} that \[H^1(\Aut(\Fu/F), G_{\ad}(\Fu)) \cong H^1(\Aut(\Fu'/F'), G_{\ad}'(\Fu'))\]
as pointed sets if the fields $F$ and $F'$ are $m$-close using the work of Kottwitz (\cite{Kot16}). Using the ideas of Debacker-Reeder \cite{DR} it is further possible to refine the above and obtain an isomorphism at the level of cocycles. This allows us to construct ``compatible" Frobenius morphisms $\sigma^*$ and $\sigma'^*$ over $F$ and $F'$ respectively and yields data $(G^*, S^*, A^*)$ where $G^* = G_\bF^{\sigma^*}$ is a connected reductive group over $F$ that is an inner form of $G$, a maximal $\Fu$-split $F$-torus $S^*$ that contains a maximal $F$-split torus $A^*$ of $G^*$, and similarly $(G'^*,S'^*,A'^*)$ over $F'$, together with a $\sigma^*$-equivariant simplicial isomorphism $\cA_{m}^*: \cA(S^*,\Fu) \rightarrow \cA(S'^*,\Fu')$. Let us explain the construction of $\sigma^*$ in more detail. 

 Let $\bcC$ be an $\sigma$-stable alcove in $\cA(S,\Fu)$.  By \cite[Corollary 2.4.3]{DR}, we have isomorphisms
\[H^1(\Aut(\Fu/F), \Omega_{\bcC, \ad}) \cong H^1(\Aut(\Fu/F), G_{\ad}(\Fu)).\]

Let $c$ be a cocycle in  $Z^1(\Aut(\Fu/F), \Omega_{\bcC, \ad})= \Omega_{\cC, \ad}$.
Let $G^*$ be the inner form of $G$ determined by $c$. Let $c(\sigma)=\bnu_\ad$ for some $\bnu_\ad \in \Omega_{\bcC, \ad}$. Write $\bnu_\ad = t_{\bea}\bz$ where $\bea\in X_*(T_{\ad})_{I_F}$, $t_\bea$ denotes the translation by $\bea$, and $\bz \in W(G,S)$. Let $L \subset F_s$ denote the finite at most $m$-ramified extension of $\Fu$ over which $G_{\Fu}$ splits. Let $n_{\bea} =\Nm_{L/\bF}(\bea(\varpi_L)) \in T_\ad(\bF)$ be a representative of $t_{\bea}$. We also fix a system of pinnings $\{x_\ba\;|\; \ba \in \breve\Phi(G,S)\}$ that is $\sigma$-stable (see \cite[Section 4.1]{BT2}); such a system of pinnings exists since $G$ is quasi-split over $F$. For each $\ba \in \breve\Delta(G,S)$, let $n_{s_\ba} = x_\ba(1) x_{-\ba}(1) x_\ba(1) \in N_G(S)(\bF)$. For $\bz = W(G,S)$, write $\bz = s_{\ba_1} \cdots s_{\ba_k}$ where $\ba_i \in \breve\Delta(G,S)$. Set $n_\bz = n_{s_{\ba_1}}\cdots n_{s_{\ba_k}}$. Then $n_\bz \in N_G(S)(\bF)$ and is independent of the choice of reduced expression of $\bz$. 

Then $\tc(\sigma)=n_{\bea}n_{\bz} \in Z^1(\Gal(\Fu/F), G_\ad(\bF)).$   Note that $\tc(\sigma) \in G_\ad(\Fu) = (\Inn(G))(\Fu)$. Let $g_\be \in T(F_s)$ be such that $j(g_\be) = n_\bea$. Let $g_\bnu = g_\be n_\bz \in G(F_s)$. Then $\tc(\sigma) = Ad(g_\bnu)$. 
Define the Frobenius action $\sigma^*$ on  element $g\in G(\Fu)$ by \[\sigma^* \cdot g = (\tc(\sigma)(\sigma \cdot g))\] (Here $\sigma \cdot g$ denotes the action of $\sigma$ on $g \in G(\Fu)$). Set $G^* = G_\bF^{\sigma^*}$. Then $G^*$ is an inner twist of $G$ whose $F$-isomorphism class in determined by the class of $\bnu_\ad$ in $H^1(\Aut(\Fu/F), \Omega_{\bcC, \ad})$. The maximal $\Fu$-split torus $S$ of $G$ gives a maximal $\Fu$-split, $\Fu$-torus $S^*$ in $G^*$. As noted in \cite[Lemma 5.4]{Gan18}. $S^*$ is defined over $F$ and with $A^*$ denoting the $F$-split torus of $G^*$ determined by $X^*(S^*)^{\sigma^*}$, $A^*$ is a maximal $F$-split torus in $G^*$. 

Now, assume $F'$ is $m$-close to $F$ and let $(G, T, B)$ correspond to $(G', T', B')$ as above. Let $\bcC' = \cA_m(\bcC)$. Then $\bcC'$ is $\sigma'$-stable and $\Omega_{\bcC, \ad} \cong \Omega_{\bcC', \ad}$. Let $\bnu_\ad'$ be the image of $\bnu_\ad$ under this isomorphism. We analogously construct $\sigma'^*$ and set $G'^* = G_{\bF'}'^{\sigma'^*}$ (See \cite[Section 5.A]{Gan18}). This yields data $(G'^*, S'^*, A'^*)$ together with a $\sigma^*$-equivariant simplicial isomorphism $\cA_m^*: \cA(S^*, \bF^*)\rightarrow \cA(S'^*, \bF'^*)$.

\begin{theorem}\label{PGCLFI}[Proposition 6.2 and Corollary 6.3 of \cite{Gan18}]
For $m \geq 1$, there exists $e>>m$ such if the fields $F$ and $F'$ are $e$-close, then  with $\bcF^*$ a facet in $\cA(S^*,\Fu)$, $\bcF'^* = \cA_{m}^*(\bcF^*)$, $\cF^* :=(\bcF^*)^{\sigma^*}$ and $\cF'^* :=(\bcF'^*)^{\sigma'^*}$, we have an isomorphism of the group schemes \[p_{m}^*: \cP_{\cF^*} \times_{\fO_F} \fO_F/\fp_F^m \rightarrow \cP_{\cF'^*} \times_{\fO_{F'}} \fO_{F'}/\fp_{F'}^m.\]
In particular,
\[\cP_{\cF^*} (\fO_F/\fp_F^m) \cong \cP_{\cF'^*} (\fO_{F'}/\fp_{F'}^m)\] as groups when $F$ and $F'$ are $e$-close.
\end{theorem}
We note that the integer $e$ that appears in this theorem (and Theorem \ref{PGCLFQ}) is the same integer that appears in the work of Chai-Yu in Theorem \ref{MTCY}, applied to a maximal torus of $G^*$. 
\subsection{The Kottwitz homomorphism}\label{KM} Let $G^*$ be a connected, reductive group over $F$. In \cite[Section 7]{Kot97}, Kottwitz constructed a surjective group homomorphism $\kappa_{G^*, \bF}: G^*(\bF) \rightarrow X^*(Z(\hat G^*))_{I_F}$ and proved that this induces a surjective homomorphism $\kappa_{G^*, F}: G^*(F) \rightarrow X^*(Z(\hat G^*))_{I_F}^{\sigma^*}$. Let $G^*(\bF)_1$ be the kernel of $\kappa_{G^*, \bF}$ and let $G^*(F)_1 = G^*(\bF)_1 \cap G^*(F)$ be the kernel of $\kappa_{G^*, F}$. 

In Section \ref{KottwitzTorus}, we will prove that the Kottwitz homomorphism admits a group theoretic section when $G^*$ is a torus. In Section \ref{KottwitzM}, we will construct a set-theoretic section of the Kottwitz homomorphism when $G^*$ is connected reductive group such that $G^*_{\ad}$ is $F$-simple and anisotropic over $F$; this will turn out to be a group theoretic section when $G^*$ is itself adjoint. These results will be used to prove Theorem \ref{MainTheorem}.

\section{Section of the Kottwitz homomorphism for tori}\label{KottwitzTorus}

Let $T$ be a torus over $F$.  We will construct a group theoretic section of $\kappa_{T, \bF}: T(\bF) \rightarrow X_*(T)_{I_F}$ that is $\sigma$-stable. This will then yield a group theoretic section of $\kappa_{T,F}: T(F) \rightarrow X_*(T)_{I_F}^\sigma$. We will then prove a comparison lemma over close local fields that will be used later in this article.

\begin{remark}
Let $\tF \subset F_s$ be the splitting extension of $T$ over $\bF$. By \cite[Section 7.2]{Kot97}, we have the following commutative diagram
\begin{equation}\label{KottTD}
\begin{tikzcd}
T(\tF) \arrow{r}{\kappa_{T,\tF}} \arrow{d}{\Nm_{\tF/\bF}}
&X_*(T) \arrow{d}{\pr}\\
 T(\bF)\arrow{r}{\kappa_{T, \bF}} &X_*(T)_{I_F}.
\end{tikzcd}\end{equation} 
For each $\blambda \in X_*(T)_{I_F}$ one may choose $\tilde\lambda \in X_*(T)$ with $\pr(\tilde\lambda) = \breve\lambda$ and choose a representative of $\blambda$ as $\Nm_{\tF/\bF}(\tilde\lambda(\varpi_\tF)) \in T(\bF)$. However, in general, such a set $\{\Nm_{\tF/\bF}(\tilde\lambda(\varpi_\tF))\; |\; \blambda \in X_*(T)_{I_F}\}$ need not form a group. Some extra work is needed to obtain representatives that form a group and is $\sigma$-stable.
\end{remark}

\subsection{$\Gamma_F$-stable representatives of $X_*(T)$: the case of induced torus}\label{Sec:IT} Let $T = Res_{L/F} \bbG_m$, where $L$ is a finite separable extension of $F$. Let $\bL = L \cap \bF$ and let $f$ be the degree of $\bL$ over $F$. Let $\tL$ be the Galois closure of $L$ in $F_s$. Let $\tF =\tL \Fu$. Fix a uniformizer $\varpi_\tL$ of $\tL$ such that $\Nm_{\tL/\bL}\varpi_\tL = \varpi_F$.

Since $T$ is induced, its cocharacter lattice admits a basis $\cB: =\{\tilde\lambda_1, \cdots, \tilde\lambda_n\}$ that is permuted simply transitively by $\Gal(\tL/F)$. Set
\[n_{\tilde\lambda_1}: = \tilde\lambda_1(\varpi_\tL).\]
For each $i>1$, there exists a unique $\gamma \in \Gal(\tL/F)$ such that $\tilde\lambda_i = \gamma(\tilde\lambda_1)$. Set
\[n_{\tilde\lambda_i}: = \gamma(n_{\tilde\lambda_1}).\]

Given $\tilde\lambda \in X_*(T)$, write $\tilde\lambda: = \sum_i c_i \tilde\lambda_i$. Set $n_{\tilde\lambda} = \prod_i n_{\tilde\lambda_i}^{c_i}$. 

\begin{lemma}\label{tlGstable}
The set $\{n_{\tilde\lambda}\;|\; \tilde\lambda \in X_*(T)\}\subset T(\tL)$ forms a group. Further, it is $\Gamma_F$-stable. 
\end{lemma}
\begin{proof}
It is clear that $\{n_{\tilde\lambda}\;|\; \tilde\lambda \in X_*(T)\}$ forms a group. To prove that it is $\Gamma_F$-stable,  it suffices to that the set $\{n_{\tilde\lambda_i}\; |\; 1 \leq i \leq n\}$ is $\Gamma_F$-stable. To see this, we simply note that since $T$ splits over $L$, we have for $\gamma \in \Gal(F_s/\tL)$, $\gamma(\tilde\lambda_i) = \tilde\lambda_i$. Also $\gamma(\varpi_\tL) = \varpi_\tL$. This finishes the proof of the lemma.
\end{proof}

\subsection{$\Gamma_F$-stable representatives of $X_*(T)$: the general case}\label{TorusLifting} Let $T$ be any torus over $F$. Following \cite[Section 7.2]{Kot97}, we choose induced tori $R$ and $S$ defined over $F$ such that
\[S \rightarrow R \xrightarrow{\psi} T \rightarrow 1\]
and we have an exact sequence of $\Gamma_F$-modules
\[X_*(S) \rightarrow X_*(R) \xrightarrow{\psi} X_*(T) \rightarrow 1.\] 
Fix a $\Gamma_F$-stable set $\{n_{\tilde\mu}\;|\; \tilde\mu \in X_*(R)\}$ (see Lemma \ref{tlGstable}).  For each $\tilde\lambda \in X_*(T)$, choose $\tilde\mu \in X_*(R)$ such that $\psi(\tilde\mu) = \tilde\lambda$. Set $n_{\tilde\lambda} = \psi(n_{\tilde\mu})$. 
\begin{lemma}\label{tlGstableg} 
The set $\{n_{\tilde\lambda}\;|\; \tilde\lambda \in X_*(T)\}$ forms a group and is $\Gamma_F$-stable.
\end{lemma}
\begin{proof}
The set $\{n_{\tilde\mu}\;|\; \tilde\mu \in X_*(R)\}$ forms a group and is $\Gamma_F$-stable by Lemma \ref{tlGstable}. Now the lemma follows using the fact that $\psi$ is a group homomorphism and is $\Gamma_F$-equivariant.
\end{proof}

\subsection{$\sigma$-stable representatives of $X_*(T)_{I_F}$}\label{sstableT}
Let $T$ be a torus over $F$ and let $\pr: X_*(T) \rightarrow X_*(T)_{I_F}$ be as in Diagram \ref{KottTD}.  Fix a $\Gamma_F$-stable set of representatives $\{n_{\tilde\lambda}\;|\; \tilde\lambda \in X_*(T)\}$ (see Lemma \ref{tlGstableg}). Let $\breve\lambda \in X_*(T)_{I_F}$ and let $\tilde\lambda \in X_*(T)$ such that $\pr(\tilde\lambda) = \breve\lambda$. Set 
\begin{equation}\label{nbli}
n_{\breve\lambda}: = \Nm_{\tF/\Fu} n_{\tilde\lambda}.
\end{equation}
\begin{lemma}\label{tlsstable}
The definition of $n_{\breve\lambda}$ in \eqref{nbli} does not depend on the choice of $\tilde\lambda$. Further, the set $\{n_{\breve\lambda}\; |\; \breve\lambda \in X_*(T)_{I_F}\}$ forms a group and is $\sigma$-stable.
\end{lemma}
\begin{proof}
Suppose $pr(\tilde\lambda') = \breve\lambda$. Then $\tilde\lambda -\tilde\lambda' \in X_*(T)(I_F)$.  Then $ \tilde\lambda -\tilde\lambda' = \sum_i \gamma_i(\tilde\mu_i) - \tilde\mu_i$ for $\gamma_i \in I_F$ and $\tilde\mu_i \in X_*(T)$. Then $\Nm_{\tF/\Fu} n_{\tilde\lambda -\tilde\lambda'} = \prod_i \Nm_{\tF/\Fu} n_{\gamma_i(\tilde\mu_i) - \tilde\mu_i}$. By Lemma \ref{tlGstableg}, we have $n_{\gamma_i(\tilde\mu_i)} = \gamma_i(n_{\tilde\mu_i})$. So \[\Nm_{\tF/\Fu}  n_{\gamma_i(\tilde\mu_i) - \tilde\mu_i} = 1\] for each $i$. So $\Nm_{\tF/\Fu} n_{\tilde\lambda} = \Nm_{\tF/\Fu} n_{\tilde\lambda'}$. Hence the definition of $n_{\breve\lambda}$ does not depend on the choice of $\tilde\lambda$. 

The set $\{n_{\breve\lambda}\; |\; \breve\lambda \in X_*(T)_{I_F}\}$ forms a group because the set $\{n_{\tilde\lambda}\; |\; \tilde\lambda \in X_*(T)\}$ forms a group. 

Next, we show that the set $\{n_{\breve\lambda}\; |\; \breve\lambda \in X_*(T)_{I_F}\}$ is $\sigma$-stable. Let $\breve\lambda \in X_*(T)_{I_F}$. Fix a lift $\tilde\sigma$ of $\sigma$ to $\Gamma_F$ and a $\tilde\lambda \in X_*(T)$ such that $\pr(\tilde\lambda) = \breve\lambda$. Then $\sigma(\breve\lambda) = \pr(\tilde\sigma(\tilde\lambda))$.  Further, 
\[\sigma(n_{\breve\lambda}) =  \Nm_{\tF/\Fu} \tilde\sigma(n_{\tilde\lambda}) = \Nm_{\tF/\Fu} n_{\tilde\sigma(\tilde\lambda)} =n_{\sigma(\breve\lambda)}.\]
Here, the second equality uses Lemma \ref{tlGstableg} and the third equality follows from \eqref{nbli}. 
\end{proof}
\subsection{A comparison lemma for tori over close local fields} We will use the notation set up in Section \ref{CY}.

\begin{lemma}\label{NMTCLF} Let $\cT$ as above and for $m \geq 1$, let $T_m  = \Ker(\cT(\fO_F) \rightarrow \cT(\fO_F/\fp_F^m))$. Let $e \geq m$ be as in Theorem \ref{MTCY}. If $F$ and $F'$ are $e$-close, we have an isomorphism \[\cT_m: T(F)/T_m \rightarrow T'(F')/T_m'.\]
\end{lemma}
\begin{proof} By Lemma \ref{tlsstable}, the exact sequence $1 \rightarrow T(F)_1 \rightarrow T(F) \xrightarrow{\kappa_{T,F}} X_*(T)_{I_F}^\sigma \rightarrow 1$ splits and we have an isomorphism of groups
 \begin{align}\label{torusiso}
T(F)_1 \times X_*(T)_{I_F} ^\sigma&\rightarrow T(F),\nonumber 
\end{align}
which induces an isomorphism
\[T(F)_1/T_m \times X_*(T)_{I_F} ^\sigma\rightarrow T(F)/T_m.\]
 Note that $\cT(\fO_F) = T(F)_1$. By the work of Chai-Yu (recalled as Theorem \ref{MTCY}), we have an isomorphism
\[\cT(\fO_F)/T_m \rightarrow \cT'(\fO_{F'})/T_m',\]
Since $T$ splits over an at most $m$-ramified extension of $F$, the action of $\Gamma_F$ on $X_*(T)$ factors through $\Gamma_F/I_F^m$. Since the action of $\Gamma_F/I_F^m$ on $X_*(T)$ is $\Del_m$-equivariant, we have $X_*(T)_{I_F}^\sigma \cong X_*(T)_{I_{F'}}^{\sigma'}$ via $\Del_m$. The lemma is proved. 
\end{proof}

\section{Section of the Kottwitz homomorphism for reductive groups with anisotropic adjoint groups}\label{KottwitzM}
We will now construct a nice set theoretic section of the Kottwitz homomorphism for connected, reductive groups whose adjoint group is anisotropic over $F$. When the group is itself adjoint, this will turn out to be a group theoretic section. We will then prove a comparison result over close local fields for such groups.

\subsection{Section of the Kottwitz homomorphism}\label{SKHM} By the classification theorem (see \cite{Kne, BT3}), we know that a connected, reductive group whose adjoint group is $F$-simple and anisotropic over $F$ is an inner form of a quasi-split connected, reductive group $M$ with $M_{\ad} \cong \Res_{L/F} \PGL_n$ for a suitable finite separable extension $L/F$. Let $\tF$ be the Galois closure of $L\bF$ in $F_s$.

Let $\sigma$ denote the quasi-split Frobenius action on $M_\Fu$ so that the $F$-structure it yields is $M$. Let $A$ be a maximal $F$-split torus in $M$, $S$ a maximal $\bF$-split $F$ torus containing $A$ and let $T = Z_M(S)$; then $T$ a maximal torus in $M$ defined over $F$.  Let $B$ be a Borel subgroup of $M$ containing $T$. We fix a set of $\sigma$-stable representatives $\{n_{\bl_\ad}\;|\; \bl_\ad \in X_*(T_\ad)_{I_F}\}$ and $\{n_{\bl}\;|\; \bl \in X_*(T)_{I_F}\}$; such a set of representatives exists by Lemma \ref{tlsstable}. The choice of $B$ gives a set of simple roots of $\tilde\Phi(M,T)$ whose restriction to $S$ gives a set of simple roots of $\breve\Phi(M,S)$, which we denote as $\breve\Delta(M,S)$. We also fix a system of pinnings $\{x_\ba\;|\; \ba \in \breve\Phi(M,S)\}$ that is $\sigma$-stable (see \cite[Section 4.1]{BT2}); such a system of pinnings exists since $M$ is quasi-split over $F$. For each $\ba \in \breve\Delta(M,S)$, let $n_{s_\ba} = x_\ba(1) x_{-\ba}(1) x_\ba(1) \in N_M(S)(\bF)$. For $\by = W(M,S)$, write $\by = s_{\ba_1} \cdots s_{\ba_k}$ where $\ba_i \in \breve\Delta(M,S)$. Set $n_\by = n_{s_{\ba_1}}\cdots n_{s_{\ba_k}}$. Then $n_\by \in N_M(S)(\bF)$ and is independent of the choice of reduced expression of $\by$. 

\begin{lemma}\label{TWMS} Let $j:M \rightarrow M_{\ad}$ be the adjoint quotient map. 
\begin{enumerate}
\item For $\breve\lambda \in X_*(T)_{I_F}$, and $\by \in W(M,S)$, we have $\by(\bl) - \bl = \by(\bl_\ad) - \bl_\ad$, where $\bl_\ad = j(\bl)$. 
\item For each $\bl_\ad \in X_*(T_\ad)_{I_F}$ and $\by \in W(M,S)$, we have $j(n_\by) n_{\bl_\ad} j(n_\by)^{-1} = n_{\by(\bl_\ad)} \in T_\ad(\bF)$. 
\item Let $t \in T(F_s)$ with $j(t) = n_{\bl_{\ad}}$.  Then for $\by \in W(M,S)$, $t n_\by t^{-1} n_\by^{-1} = n_{\bl_{\ad} - \by(\bl_\ad)} \in T(\bF)$.
\item For $\bl \in X_*(T)_{I_F}$ and $n_\bl$ as in Lemma \ref{tlsstable}, we have $n_\bl n_\by n_\bl^{-1} n_\by^{-1} = n_{\bl - \by(\bl)}$. 
\end{enumerate}
\end{lemma}
\begin{proof} 
For (a), simply note that $\by(\bl) - \bl$ and $\by(\bl_\ad) - \bl_\ad$ belong to $ X_*(T_{\sc})_{I_F} $ and $j$ acts as identity on $X_*(T_{\sc})_{I_F}$. 

Let us prove (b). Note that $M_{\ad, \bF} = M_{\ad,\Fu}^{(1)} \times  M_{\ad,\Fu}^{(2)}  \cdots \times M_{\ad,\Fu}^{(k)} $ where $ M_{\ad,\Fu}^{(1)} \cong  M_{\ad,\Fu}^{(2)}  \cdots \cong  M_{\ad,\Fu}^{(k)} \cong \Res_{L\Fu/\Fu} \PGL_n$. 
Further, $X_*(T_\ad)_{I_F} = X_*(T_{\ad}^{(1)})_{I_F} \times X_*(T_{\ad}^{(2)})_{I_F} \times X_*(T_{\ad}^{(k)})_{I_F}$ and $W(M,S) = W(M^{(1)}, S^{(1)}) \times \cdots W(M^{(k)}, S^{(k)})$ with each $W(M^{(i)}, S^{(i)}) \cong S_n$. So it suffices to prove the lemma for $\bl_\ad \in  X_*(T_{\ad}^{(i)})_{I_F}$ and $\by \in W(M^{(i)}, S^{(i)})$. 
Since the torus $T_\ad^{(i)}$ is induced, it admits a $\bbZ$-basis  $\{\gamma(\tilde\lambda_{\ad, 1}),\gamma( \tilde\lambda_{\ad, 2}) \cdots , \gamma(\tilde\lambda_{\ad,n-1})\;|\; \gamma \in I_F\}$ permuted by $I_F$ which then yields a basis $\{\bl_{\ad, 1}, \cdots , \bl_{\ad, n-1}\}$ of $X_*(T_{\ad}^{(i)})_{I_F}$. Note that for each $\by \in W(M,S)$ and each $\gamma \in I_F$, $\by$ stabilises the lattice $\cL_\gamma = \bbZ\langle \gamma(\tilde\lambda_{\ad, 1}),\gamma( \tilde\lambda_{\ad, 2}), \cdots , \gamma(\tilde\lambda_{\ad,n-1})\rangle$.  Also note that the $\cL_\gamma, \gamma \in \Gal(\tF/\bF)$, are disjoint.   
By Section \ref{Sec:IT}, we have for a suitable choice of $\varpi_{L\bF}$ that $n_{\tilde\mu_{\ad}} = \tilde\mu_{\ad}(\gamma(\varpi_{L\bF}))$ for every $\tilde\mu_{\ad} \in \cL_\gamma$ and every $\gamma \in I_F$. 
Further, since the set of representatives $ \{n_{\blambda_{\ad}}\;|\; \blambda_{\ad} \in X_*(T_{\ad})_{I_F}\}$ forms a group by Lemma \ref{tlsstable}, it suffices to check (b) for $\bl_{\ad, r}, 1 \leq r \leq n-1$. 
Then $j(n_\by) n_{\bl_{\ad,r}}j(n_\by)^{-1} =  j(n_\by)\left(\prod_{\gamma \in I} \gamma(\tilde\lambda_{\ad,r})(\gamma(\varpi_{L\bF})) \right)j(n_\by)^{-1}=  \prod_{\gamma \in I} \gamma(\by(\tilde\lambda_{\ad,r}))(\gamma(\varpi_{L\bF})) = n_{\by(\bl_{\ad,r})}$.

Let us prove (c). It suffices to prove (c) for a chosen $t \in T(F_s)$ with $j(t) = n_{\bl_\ad}$ since any other $t'$ with this property will differ from $t$ by an element of $Z(M)(F_s)$. Also, it suffices to prove (c) for $\bl_{\ad, r}, 1 \leq r \leq n-1$. The sequence $1 \rightarrow Z(M) \rightarrow M \rightarrow M_\ad \xrightarrow{j} 1$ induces an exact sequence
\[ X_*(Z(M))_{I_F} \rightarrow X_*(T)_{I_F} \xrightarrow{j} X_*(T_\ad)_{I_F}.\]
The group $X_*(T_\ad)_{I_F}/j(X_*(T)_{I_F})$ is finite. Let $k$ be the smallest positive integer such that $k \bl_{\ad,r} \in j(X_*(T)_{I_F})$. Let $\bl_r\in X_*(T)_{I_F}$ such that $j(\bl_r) =k \bl_{\ad,r}$. Note that $n_{\bl_{\ad,r}} =  \prod_{\gamma \in I} \gamma(\tilde\lambda_{\ad,r})(\gamma(\varpi_{L\bF}))$. Fix $k$-th roots of $\gamma(\varpi_{L\bF})$ in $F_s$ and set 
$t_r =  \prod_{\gamma \in I} \gamma(\tilde\lambda_{r})(\gamma(\varpi_{L\bF})^{1/k})\in T(F_s)$. Then $j(t_r) = n_{\bl_{\ad, r}}$. Now, 
$t_r n_\by t_r^{-1} n_\by^{-1} = \prod_{\gamma \in I} \gamma(\tilde\lambda_{r} - \by(\tilde\lambda_r))(\gamma(\varpi_{L\bF})^{1/k})$. Note that $\tilde\lambda_{r} - \by(\tilde\lambda_r) = k(\tilde\lambda_{\ad, r} - \by(\tilde\lambda_{\ad, r}))$ because $\tilde\lambda_{r} - \by(\tilde\lambda_r)$ and $ k(\tilde\lambda_{\ad, r} - \by(\tilde\lambda_{\ad, r}))$ belong to $X_*(T_{\sc})$ and $j$ acts as identity on $X_*(T_{\sc})$. 

Then  
\begin{align*}
t_r n_\by t_r^{-1} n_\by^{-1} &= \prod_{\gamma \in I} \gamma(\tilde\lambda_{r} - \by(\tilde\lambda_r))(\gamma(\varpi_{L\bF})^{1/k})\\
& =  \prod_{\gamma \in I} k\cdot\gamma(\tilde\lambda_{\ad, r} - \by(\tilde\lambda_{\ad,r}))(\gamma(\varpi_{L\bF})^{1/k})\\
& =  \prod_{\gamma \in I} \gamma(\tilde\lambda_{\ad, r} - \by(\tilde\lambda_{\ad,r}))(\gamma(\varpi_{L\bF}))\\
& = n_{\bl_{\ad} - \by(\bl_\ad)}.
\end{align*} 

Now (d) follows from (c), the fact that $j(n_\bl) = n_{\bl_\ad}$, and (a). 
\end{proof}

Let $\Omega_{\brM}  = X^*(Z(\hat M))_{I_F}$ and $\Omega_{\brM, \ad}  = X^*(Z(\hat M_\ad))_{I_F}$. We know that the $F$-isomorphism classes of inner twists of $M$ is parametrized by the pointed cohomology set $H^1(\Gamma_F/I_F, \Omega_{\brM, \ad}) \cong (\Omega_{\brM, \ad})_{\sigma}$. 

Now, $\Omega_{\brM, \ad} = \Omega_{\brM^{(1)}, \ad} \times  \Omega_{\brM^{(2)}, \ad }\cdots \times  \Omega_{\brM^{(k)}, \ad }$, $\sigma$ permutes these factors transitively, and $(\Omega_{\brM, \ad})_{\sigma} \cong \bbZ/n\bbZ$. Further, the natural surjection $ \Omega_{\brM^{(1)} , \ad} \rightarrow  (\Omega_{\brM, \ad})_{\sigma}$ is an isomorphism. 
Note that $\Omega_{\brM^{(1)}, \ad}  \hookrightarrow X_*(T_{\ad}^{(1)})_{I_F} \rtimes W(M^{(1)}, S^{(1)})$ and $W(M^{(1)},S^{(1)} )\cong S_n$. Let $\bnu_\ad  = t_{\breve\eta_\ad}\bz \in \Omega_{\brM^{(1)}, \ad} $ with $\bz =\bs_1 \cdots \bs_{n-1}$ (Here $M_\ad^{(1)}$ is adjoint of type $A_{n-1}$ and we have used the labelling of the finite simple roots as in Bourbaki \cite{Bou02}). The group we are interested in this section is determined by the image of $\bnu_\ad $ in  $ (\Omega_{\brM, \ad})_{\sigma}$. Let $\sigma^*$ be the Frobenius morphism associated to $\bnu_\ad$ as in Section \ref{Gan18IF}, and consider the group $M^* = M_\bF^{\sigma^*}$.

We will construct a section of the Kottwitz homomorphism $\kappa_{M^*, F}: M^*(F)\rightarrow \Omega_{M^*}$ (see Section \ref{KM} for notation). Let  $\Omega_M = \Omega_{\brM}^\sigma$ and $\Omega_{M^*} =\Omega_{\brM}^{\sigma^*}$. Similarly define  $\Omega_{M, \ad}$ and $\Omega_{M^*, \ad}$. 
\begin{lemma}\label{MM*} For $\btau \in \Omega_{\brM}$, we have $\sigma^*(\btau) = \sigma(\btau)$. In particular, $\Omega_M = \Omega_{M^*}$ and $\Omega_{M, \ad} = \Omega_{M^*, \ad} \cong \bbZ/n\bbZ$.
\end{lemma}
\begin{proof} Let $\btau=t_{\blambda} \by$. Then $\sigma^*(\btau) = t_{\sigma^*(\blambda)}t_{\bea - y'(\bea)}\by'$, where $\by'=Ad(\bz)(\sigma(\by))$. 
To prove that $\sigma^*(\btau) = \sigma(\btau)$, we need to prove that $\sigma^*(\blambda) + \bea - y'(\bea) = \sigma(\lambda)$. 
Let $\btau_\ad = j(\btau) = t_{\blambda_\ad} \by$. 
Then, since $\Omega_{\brM, \ad}$ is abelian, we have $\sigma^*(\blambda_\ad) - \sigma(\blambda_\ad)= \bea - y'(\bea)$ as elements of $X_*(T_\ad)_{I_F}$. 
Now, 
\begin{align*}
    \bea - \by'(\bea)&= \sigma^*(\blambda_\ad) - \sigma(\blambda_\ad)= \bz(\sigma(\blambda_\ad)) -\sigma(\blambda_\ad)= \bz(\sigma(\blambda))-\sigma(\blambda) = \sigma^*(\blambda) - \sigma(\blambda).
\end{align*}
In the above, the third equality is by Lemma \ref{TWMS}(a). This proves that $\sigma^*(\btau) = \sigma(\btau)$. The rest of the lemma is obvious. 
\end{proof}

The group $j(\Omega_{M}) \subset \Omega_{M, \ad}$ is cyclic. Let $\btau_0 \in \Omega_{\brM}$ be such that $j(\btau_0)$ is a  generator of $j(\Omega_{M})$. Then $j(\btau_0) = \bnu_\ad^r\sigma(\bnu_\ad)^r \cdots \sigma^{k-1}(\bnu_\ad)^r$ for a suitable $r$.  Write $\btau_0 = t_{\bl_0} \by_0 \in X_*(T)_{I_F} \rtimes W(M,S)$. Then $\by_0 = \bz^r \sigma(\bz^r) \cdots \sigma^{k-1}(\bz)^r$. Let $n_{\by_0} = n_\bz^r\sigma(n_\bz^r) \cdots \sigma^{k-1}(n_\bz^r)$. Set $n_{\btau_0} = n_{\bl_0} n_{\by_0}$. 
Since 
 $X_*(T_{\sc})_{I_F} \cap X_*(Z(M))_{I_F}$ is trivial, we get an exact sequence
\[ X_*(Z(M))_{I_F} \xrightarrow{\phi} \Omega_\brM \xrightarrow{j} \Omega_{\brM, \ad}\]
which then yields 
\[ X_*(Z(M))_{I_F}^{\sigma} \xrightarrow{\phi} \Omega_\brM^{\sigma} \xrightarrow{j} \Omega_{\brM, \ad}^{\sigma}.\]

Now, given $\btau \in \Omega_{M}$, there exist $\breve \mu \in X_*(Z(M))_{I_F}^\sigma$ and $s \in \bbZ$ divisible by $r$ such that $\btau = \phi(\breve \mu) +s\btau_0$. Set $n_{\btau} = \phi(n_\bmu) n_{\btau_0}^{s}$ where $n_\bmu \in Z(M)(\bF)^\sigma = Z(M)(F)$.

\begin{proposition}\label{GTS}
 Let $\btau \in \Omega_{M^*} = \Omega_M$. Then $\sigma^*(n_{\btau}) = n_{\btau}$. In particular, $n_{\btau} \in M^*(F)$ and $\tp: \Omega_{M^*} \rightarrow M^*(F), \btau \rightarrow n_{\btau},$ is a (set-theoretic) section of $\kappa_{M^*, F}$. If we additionally assume that $M^*$ is adjoint, then $\tp$ is a group-theoretic section.
\end{proposition} 
\begin{proof}It suffices to prove that $\sigma^*(n_{\btau_0}) = n_{\btau_0}$ since for $ \breve \mu \in X_*(Z(M))_{I_F}^\sigma$, $n_\bmu \in Z(M)(\bF)^\sigma$ and $\sigma^* = Ad(g_\bnu)\circ \sigma$ with $g_\bnu \in M(F_s)$. 

Now $\sigma^*(\btau_0) = \btau_0$ and $\sigma(\btau_0) = \btau_0$ imply that $\bz(\bl_0) + \be_\ad -\by_0(\be_\ad) = \bl_0$ and $\bz\by_0\bz^{-1} = \by_0$. Note that $\sigma(n_\btau) = n_\btau$ by Lemma \ref{tlsstable} and the construction of $n_\by$. Now $\sigma^*(n_{\btau_0}) = n_{\btau_0}$ if and only if $Ad(g_{\bnu})(n_{\bl_0} n_{\by_0}) = n_{\bl_0} n_{\by_0}$. By Lemma \ref{TWMS}, $Ad(g_{\bnu})(n_{\bl_0}) = n_{\bz(\bl_0)}$ and $Ad(g_{\bnu})( n_{\by_0})= n_{\be_\ad - \by_0(\be_\ad)} n_{\by_0}$. Hence \[Ad(g_{\bnu})(n_{\bl_0} n_{\by_0}) =  n_{\bz(\bl_0)}n_{\be_\ad - \by_0(\be_\ad)} n_{\by_0} = n_{\bz(\bl_0) +\be_\ad - \by_0(\be_\ad)}n_{\by_0} = n_{\bl_0}n_{\by_0},\] proving that $n_{\tau_0} \in M^*(F)$. Evidently $\kappa_{M^*, F}(n_{\btau}) = \btau$ for each $\btau \in \Omega_{M^*}$.

Before proving the last statement, we observe the following about the section $\tp: \Omega_{M^*} \rightarrow M^*(F)$. First note that $j(\btau_0)$ has order $n/r$ and hence $({n/r})\btau_0 = \phi(\bmu)$ for a suitable $\bmu \in X_*(Z(M))_{I_F}$.  and  $n_{\btau_0}^{n/r} = n_{\phi(\bmu)} n_{\by_0}^{n/r}= n_{\phi(\bmu)} n_\bz^n \sigma(n_\bz)^n \cdots \sigma^{k-1}(n_{\bz})^n$. The element $ n_\bz^n \sigma(n_\bz)^n \cdots \sigma^{k-1}(n_{\bz})^n = \ba^\vee(-1) \in M^*(F)_1$ for a suitable $\ba \in \breve\Phi(M,S)$.  One can explicitly calculate $\ba^\vee$ and note that $\ba^\vee(-1) \in Z(M^*)(F)$ (for example, this easily follows from \cite[Theorem C]{AH17}).  

Let $\btau_1, \btau_2 \in \Omega_{M^*}$. If $j(\btau_1+\btau_2) \neq 1$, we have by construction that $\tp(\btau_1+\btau_2) = \tp(\btau_1) \tp(\btau_2)$. If $j(\btau_1 +\btau_2)=1$, then $\tp(\btau_1+\btau_2) = \tp(\btau_1) \tp(\btau_2)(\ba^\vee(-1))^l$ for a suitable $l \geq 1$. If $M^*$ is adjoint then $\ba^\vee(-1)=1$ and hence the section $ \tp$ is a group-theoretic section. 
\end{proof}
\subsection{An isomorphism over close local fields} We keep the notation of Section \ref{SKHM}.  Then $M^*_{\der}$ is anisotropic over $F$ and $\cB(M^*,F)$ is a single point $x$. The subgroup $M^*(F)_1$ is the unique parahoric subgroup of $M^*(F)$ attached to the point $x$. Let us denote the underlying group scheme as $\cM^*$. More precisely, the generic fiber of $\cM^*$ is $M^*$ and $\cM^*(\fO_F) = M^*(F)_1$. Let $M_m^* = \Ker(\cM^*(\fO_F) \rightarrow \cM^*(\fO_F/\fp_F^m))$. Since $M^*(F)_1$ is the unique parahoric subgroup of $M^*(F)$, we see that $M_m^*$ is normal in $M^*(F)$.

Let $m\geq 1$ be such that $M$ splits over an at most $m$-ramified extension of $F$. Let $e>>m$ be as in Theorem \ref{PGCLFQ}, and let $F'$ be another non-archimedean local field that is $e$-close to $F$. Let $\sigma'$ be the corresponding quasi-split Frobenius morphism over $F'$ as in \ref{Gan18QS} and let $(M', T', B')$ be the corresponding groups over $F'$. Let $\breve\Delta(M',S')$ be the set of simple roots of $\breve\Phi(M',S')$ (determined by $B'$). We also fix a compatible system of pinnings $\{x_{\ba'}\;|\; \ba \in \breve\Phi(M',S')\}$ that is $\sigma'$-stable as in \cite[Section 4.A.1]{Gan18}. For each $\ba' \in \breve\Delta(M',S')$, let $n_{s_{\ba'}} = x_{\ba'}(1) x_{-\ba'}(1) x_{\ba'}(1) \in N_{M'}(S')(\bF')$. Using this, we obtain a representative $n_{\by'} \in N_{M'}(S')(\bF')$ for each $\by' \in W(M',S')$. 

We fix a set of $\sigma'$-stable representatives $\{n_{\bl_\ad'}\;|\; \bl_\ad' \in X_*(T_{\ad}')_{I_{F'}}\}$ and $\{n_{\bl'}\;|\; \bl' \in X_*(T')_{I_{F'}}\}$; such a set of representatives exist by Lemma \ref{tlsstable}. 

As recalled in Section \ref{Gan18IF}, we have 
\[\Omega_{\brM, \ad}\cong \Omega_{\brM', \ad}\]
Let $\bnu_\ad' \in \Omega_{\brM', \ad}$ be the image of $\bnu_\ad$ under this isomorphism.  Write $\bnu_\ad'  = t_{\breve\eta_\ad'}\bz' \in \Omega_{\brM'^{(1)}, \ad} $ with $\bz' =\bs_1' \cdots \bs_n'$. Let $\sigma'^*$ be the corresponding Frobenius morphism associated to $\bnu_\ad'$ and let $M'^* = M_{\bF}'^{\sigma'^*}$.

We have the following proposition.
\begin{proposition}\label{MmCLF} For $m \geq 1$, let $e \geq m$ be as in Theorem \ref{PGCLFI}. If $F$ and $F'$ are $e$-close, we have an isomorphism
\[M^*(F)/M_m^* \xrightarrow{\cong} M'^*(F')/M_m'^*.\]
\end{proposition}

\begin{proof} By Theorem \ref{PGCLFI} we have
\[M^*(F)_1/M_m^* \cong M'^*(F')_1/M_m'^*.\]
Further, by Lemma \ref{MM*}, we have $\Omega_M = \Omega_{M^*}$, and since the isomorphisms $X_*(T) \rightarrow X_*(T')$ and $X_*(T_{\sc}) \rightarrow X_*(T_{\sc}')$ are $\Del_m$-equivariant (see \cite[Lemma 3.4]{Gan18}), we have $\Omega_M \cong \Omega_{M'}$. Hence
\[ \Omega_{M^*} \cong \Omega_{M'^*}.\]
We identify these groups via these isomorphisms, and, to prove the lemma, we need to prove that the group extensions
\[
\begin{tikzcd}
 & M^*(F)/M^*_m\arrow[dr,"\phi"] \\
1\rightarrow M^*(F)_1/M^*_m \arrow[ur] \arrow[dr] &&\Omega_{M^*} \rightarrow 1\\
&M'^*(F')/M_m'^*\arrow[ur,"\phi'"]
\end{tikzcd}
\]
are equivalent.
To do this, it suffices to show that there exist set-theoretic sections $p:\Omega_{M^*} \rightarrow M^*(F)/M_m$ and $p':\Omega_{M'^*} \rightarrow M'^*(F')/M_m'^*$ such that 
\begin{enumerate}
\item with $\psi = Inn \circ p$ and $\psi' = Inn \circ p'$ from $\Omega_{M^*}  \rightarrow Aut(M^*(F)_1/M_m^*)$, we have $\psi = \psi'$,
\item with $\chi, \chi': \Omega_{M^*}  \times \Omega_{M^*}   \rightarrow M^*(F)_1/M_m^*$ given by $\chi(\btau_1,\btau_2) = p(\btau_1+\btau_2)p(\btau_2)^{-1} p(\btau_1)^{-1}$ and $\chi'(\btau_1,\btau_2) = p'(\btau_1+\btau_2)p'(\btau_2)^{-1} p'(\btau_1)^{-1}$, we have $\chi = \chi'$.
\end{enumerate}
Consider the set theoretic section $\Omega_{M^*} \xrightarrow{\tp} M^*(F)$ in Proposition \ref{GTS} and let $p$ be the projection of this section to $M^*(F)/M_m^*$. Let $\btau_0'$ be the image of $\btau_0$ under the isomorphism $\Omega_{M^*} \rightarrow \Omega_{M'^*}$. Then $j(\btau_0')$ generates $j(\Omega_{M'}) \subset \Omega_{M', \ad}$. With $r$ as in  the paragraph preceding Proposition \ref{GTS},  we have $j(\btau_0') = \bnu_\ad'^r\sigma'(\bnu_\ad')^r \cdots \sigma'^{r-1}(\bnu_\ad')^r$.  Write $\btau_0' = t_{\bl_0'} \by_0' \in X_*(T')_{I_F} \rtimes W(M',S')$. Note that under the isomorphism $X_*(T)_{I_F} \rightarrow X_*(T')_{I_{F'}}$, $\bl_0 \rightarrow \bl_0'$. Further $\by_0' = \bz'^r \sigma'(\bz'^r) \cdots \sigma'^{r-1}(\bz')^r$. 
Let $n_{\by_0'} = n_{\bz'}^r\sigma'(n_{\bz'}^r) \cdots \sigma'^{r-1}(n_{\bz'}^r)$. Set $n_{\btau_0'} = n_{\bl_0'} n_{\by_0'}$. 
Given $\btau' \in \Omega_{M'^*}$, we may write $\btau' = \phi'(\breve \mu') +s\btau_0'$ where $\breve\mu'$ maps to $\breve \mu$ under the isomorphism $X_*(Z(M'))_{I_{F'}}^{\sigma'} \cong X_*(Z( M))_{I_F}^\sigma$. Set $n_{\btau'} = \phi(n_{\bmu'}) n_{\btau_0'}^{s}$ where $n_\bmu \in Z(M')(\bF')^{\sigma'}$. Again by Proposition \ref{GTS}, $n_{\btau'} \in M'^*(F')$. We have constructed a section $\tp': \Omega_{M'^*} \rightarrow M'^*(F')$. Let $p'$ be the projection of this section to $M'^*(F')/M_m'^*$. 

Now let us prove that the sections $p$ and $p'$ constructed in the preceding paragraph satisfy (a) and (b). 

To see (a), it suffices to prove that 
\[\begin{tikzcd}
M^*(F)_1/M_m^*\arrow{r}{\cong} \arrow{d}{Inn(n_{\btau_0})}
&M'^*(F')_1/M_m'^*\arrow{d}{Inn(n_{\btau_0'})}\\
M^*(F)_1/M_m^*\arrow{r}{\cong} &M'^*(F')_1/M_m'^*
\end{tikzcd}\]
is commutative. Note that \[(\sigma^*)^k = Ad(g_\bnu \sigma(g_\bnu) \cdots \sigma^{k-1}(g_\bnu)) \circ \sigma^k = Ad(n_{\bnu_\ad} \sigma(n_{\bnu_\ad}) \cdots \sigma^{k-1}(n_{\bnu_\ad})) \circ \sigma^k\] and hence $(\sigma^*)^{rk} = Ad(j(n_{\btau_0})) \circ \sigma^{rk} =  Ad(n_{\btau_0}) \circ \sigma^{rk}$. Since $\sigma^*$ (and hence $(\sigma^*)^{rk}$) fixes $M^*(F)_1$ pointwise, and since $n_{\btau_0} \in M^*(F)$ (and hence normalizes $M^*(F)_1$), it follows that $\sigma^{rk}$ also stabilizes $M^*(F)_1$. Hence, to prove that the above diagram is commutative, it suffices to prove that the following diagram 
\begin{equation}\label{MMm*}
\begin{tikzcd}
M^*(F)_1/M_m^*\arrow{r}{\cong} \arrow{d}{\sigma^{rk}}
&M'^*(F')_1/M_m'^*\arrow{d}{\sigma'^{rk}}\\
M^*(F)_1/M_m^*\arrow{r}{\cong} &M'^*(F')_1/M_m'^*
\end{tikzcd}
\end{equation}
is commutative. Let $\bP$ be a parahoric subgroup of $M(\bF)$($=M^*(\bF)$) and let $\bP'$ be the corresponding parahoric subgroup of $M'(\bF')$ (see Section \ref{Gan18QS}). Then by \cite[Theorem 4.5]{Gan18}, we have that $\bP/\bP_m \cong \bP'/\bP'_m$,  and by \cite[Proposition 4.10]{Gan18}, this isomorphism is $\sigma$-equivariant. Now, using the facts that $\bP \cap M^*(F) = M^*(F)_1$, $\bP_m \cap M^*(F) = M^*_m$ and similarly that $\bP' \cap M'^*(F') = M'^*(F')_1$, $\bP'_m \cap M'^*(F') = M'^*_m$, the commutativity of diagram \ref{MMm*} follows. This finishes the proof of (a).

Let us prove (b). As noted in the proof of \ref{GTS}, the element $ n_\bz^n \sigma(n_\bz)^n \cdots \sigma^{k-1}(n_{\bz})^n = \ba^\vee(-1) \in Z(M^*)(F) \cap M^*(F)_1$. Further, for $\btau_1, \btau_2 \in \Omega_{M^*}$, we have  $\tp(\btau_1+\btau_2) = \tp(\btau_1) \tp(\btau_2)$ if $j(\btau_1+\btau_2) \neq 1$, and $\tp(\btau_1+\btau_2) = \tp(\btau_1) \tp(\btau_2)(\ba^\vee(-1))^l$ for a suitable $l$ if $j(\btau_1 +\btau_2)=1$.  
Now, under the isomorphism $\Omega_{M^*} \rightarrow \Omega_{M'^*}$, $(n/r)\btau_0 \rightarrow (n/r)\btau_0'$ and hence $n_{\btau_0'}^{n/r} = n_{\phi'(\bmu')} n_{\by_0'}^{n/r}= n_{\phi'(\bmu')} n_{\bz'}^n \sigma(n_{\bz'})^n \cdots \sigma^{k-1}(n_{\bz'})^n$. With $\ba \rightarrow \ba'$ under the isomorphism $\breve\Phi(M,S) \rightarrow \breve\Phi(M', S')$, we have $ n_{\bz'}^n \sigma(n_{\bz'})^n \cdots \sigma^{k-1}(n_{\bz'})^n = \ba'^\vee(-1)$. In particular, $\ba^\vee(-1) \mod M^*_m \rightarrow \ba'^\vee(-1) \mod M'^*_m$. Hence the sections $p,p'$ satisfy (b).

This finishes the proof of the proposition.
\end{proof}

\section{The Kazhdan isomorphism for general connected reductive groups}\label{MT}
Let $G^*$ be a connected reductive group over $F$, $S^*$ a maximal $\Fu$-split, $F$-torus in $G^*$ and $A^*$ a maximal $F$-split torus in $G^*$ contained in $S^*$. Assume $F'$ is $e$-close to $F$ where $e$ is as in Theorem \ref{PGCLFI}. Let $(G'^*,S'^*,A'^*)$ correspond to $(G^*,S^*,A^*)$ as explained in Section \ref{Gan18IF} and let $\cA_m^*:\cA(S^*,\Fu) \rightarrow \cA(S'^*,\Fu')$ be the $\sigma^*$-equivariant simplicial isomorphism discussed there. 

 Let $M^*= C_{G^*}(A^*)$. Let $v$ denote a special vertex in the closure of the alcove $\bcC^{\sigma^*}$ of the apartment $\cA(A^*,F)$. Let $K^*$ denote the parahoric subgroup of $G^*(F)$ attached to this special vertex. 

Let $\cK^*$ be the smooth affine $\fO_F$-group scheme underlying $K^*$ and let $K^*_m = \Ker(\cK^*(\fO_F) \rightarrow \cK^*(\fO_F/\fp_F^m))$.  Let $v'$ denote the special vertex of $\cA(A'^*,F')$ corresponding to $v$ (under the isomorphism $\cA_m^*: \cA(A^*, F)\rightarrow \cA(A'^*, F')$) and $\cK'^*, K'^*, K_m'^*$ be the corresponding objects over $F'$.

Consider the Hecke algebra $\cH(G^*(F), K_m^*)$. The goal of this section is to prove the following theorem.

\begin{theorem}\label{MainTheorem} Let $m\geq 1$ and let $e\geq m$ as in Theorem \ref{PGCLFI}. There exists $l \geq e$ such that for any two non-archimedean local fields $F$ and $F'$ that are $l$-close, there is a map
\[h_m^*: \cH(G^*(F), K_m^*) \rightarrow \cH(G'^*(F'), K_m'^*)\]
that is an isomorphism of $\bbC$-algebras. 
\end{theorem}
In the special case where $G^*$ is an inner form of $\GL_n$, this theorem is due to Badulescu (see \cite{Bad02}). 

In \cite{HR10}, Haines and Rostami establish the Cartan decomposition of $G^*(F)$ with respect to $K^*$.
\begin{theorem}[Theorem 1.0.3 of \cite{HR10}]\label{CD}We have a bijection
\[K^*\backslash G^*(F)/K^* \rightarrow W(G^*,A^*) \backslash \Omega_{M^*}.\]
\end{theorem}
 We have the following proposition.
 \begin{proposition}\label{CWG} Suppose $F$ and $F'$ are $e$-close, where $e$ is as in Theorem \ref{PGCLFI}. With $(G^*,S^*,A^*) \rightarrow (G'^*,S'^*,A'^*)$ as in Section \ref{Gan18IF}, arising from data $(G,T,B) \rightarrow (G',T',B')$, we have isomorphisms
 \begin{enumerate}
 \item $W(G^*,A^*) \cong W(G'^*,A'^*)$.
 \item With $M'^* = C_{G'^*}(A'^*)$, we have $\Omega_{M^*} \cong\Omega_{M'^*}$. 
 \item With $K'^*$ corresponding to $K^*$ as above, we have a bijection
 \[ K^* \backslash G^*(F)/K^* \rightarrow K'^*\backslash G'^*(F') /K'^*.\]
 \end{enumerate}
 \end{proposition}
 \begin{proof} We begin with a general observation. By \cite[Lemma 6.1.2]{HR10}, we have
 \[W(G^*,A^*) \cong \left(W(G^*,S^*)/W(M^*,S^*)\right)^{\sigma^*}.\]
Let $A_{0}$ be the torus in $G_\bF$ that is the image 
 of $A$ under the isomorphism $G^*_\bF \rightarrow G_\bF$. Let $M = C_G(A_{0})$. We claim that $A_{0}$ is $F$-split, that $M$ is an $F$-Levi subgroup of $G$, and $M^*$ is an inner twist of $M$. Let $P^*$ be the minimal $F$-parabolic subgroup with Levi component $M^*$ such that with $P$ denoting the image of $P^*_\bF$ under the isomorphism $G^*_\bF \rightarrow G_\bF$, $P$ contains $B$.  Note that apriori, $P$ is defined over $\bF$. We have $T \subset B \subset \sigma(P)$ since $T, B$ are $\sigma$-stable. Now since $\sigma(P)$ and $P$ are $G(F_s)$-conjugate and both contain $B$ we have $\sigma(P)=P$, $\sigma(M) = M$ and hence $P, M$ are defined over $F$. Since $P^*$ and  $M^*$ is defined over $F$, we have $\sigma^*(P^*) = P^*$ and $ \sigma^*(M^*) = M^*$. This implies that $\sigma(P_\bF) = \Ad(n_\bz^{-1})(P_\bF), \sigma(M_\bF) = \Ad(n_\bz^{-1})(M_\bF)$, where $n_\bz$ is as in Section \ref{Gan18IF}.  This implies that $n_\bz^{-1}$ normalizes $P$ and $M$ and hence $n_\bz^{-1} \in M(F_s)$. 
 
 Note that $X_*(A_{0}) = \{\bl \in X_*(S)\;|\; \bz(\sigma(\bl)) = \bl\}$. Since $n_\bz^{-1} \in M(F_s)$ and $M$ centralizes $A_0$ we see that for each $\bl \in X_*(A_0)$, $\bz(\bl) = \bl$. Hence $X_*(A_{0}) = \{\bl \in X_*(S)\;|\; \bz(\bl) = \bl,\,\; \sigma(\bl) = \bl\}$. In particular $A_0$ is $F$-split and $A_0 \subset A$. With $M'^* = C_{G'^*}(A'^*)$, and with $A_0', M'$ analogous objects over $F'$, we again have that $A_0'$ is $F'$-split, $M' = C_{G'}(A')$ is an $F'$-Levi subgroup of $G$ and $X_*(A_0') =\{\bl' \in X_*(S')\;|\; \bz'(\bl') = \bl',\,\; \sigma'(\bl') = \bl'\}$, where $\bz'$ is the image of $\bz$ under the isomorphism $W(G, S) \cong W(G', S')$. 
 
 Let us prove (a). Note that $W(G^*, S^*) \cong W(G'^*, S'^*)$ and $W(M^*, S^*) \cong W(M'^*, S'^*)$ and these isomorphisms are $\sigma^*$-equivariant since $W(G, S)\cong W(G', S')$, $W(M, S) \cong W(M', S')$ and these isomorphisms are $\sigma$-equivariant. Hence (a) is proved.
 
 Let us prove (b). We have $\Omega_{M^*} = \Omega_M$ by Lemma \ref{MM*}. Further, $\Omega_{\brM} \cong \Omega_{\brM'}$ and this isomorphism is $\sigma$-equivariant by \cite[Lemma 3.4]{Gan18}. Hence $\Omega_M \cong \Omega_\brM^\sigma \cong \Omega_{\brM'}^{\sigma'}\cong \Omega_{M'}$ and hence (b) holds.

(c) holds by (a), (b) and Theorem \ref{CD}. 
\end{proof}
The alcove $\bcC^{\sigma^*}$ together with the choice of special vertex $v$ determines a set of simple roots $\Delta_0$ of $\Phi(G^*, A^*)$. 
Note that $X_*(A^*) \otimes \bbR \cong \Omega_{M^*} \otimes \bbR$ and the natural $\bbZ$-bilinear pairing $\langle \cdot, \cdot \rangle: X^*(A^*) \times X_*(A^*) \rightarrow \bbZ$ extends to an $\bbR$-bilinear pairing 
 \begin{align}\label{pairing}
\langle \cdot, \cdot \rangle: (X^*(A^*) \otimes \bbR) \times (\Omega_{M^*} \otimes \bbR) \rightarrow \bbR.
 \end{align}  
 Let $\Omega_{M^*, +}$ be the set of dominant elements of $\Omega_{M^*}$ (see \cite[Section 5.2]{Ros}), that is, $\Omega_{M^*, +} = \{\tau \in \Omega_{M^*}\;|\; \langle a, \tau\rangle \geq 0 \; \forall \; a \in \Delta_0\}$. Let $\Phi^+(G^*, A^*)$ denote the set of positive roots in $\Phi(G^*, A^*)$. 
 
 Let $dg$ be a Haar measure on $G^*(F)$ so that $vol(K_m^*, dg)=1$. 
\begin{proposition}\label{simpleconv} Let $\tp: \Omega_{M^*} \rightarrow M^*(F), \tau \rightarrow n_\tau$, be the section of the Kottwitz homomorphism $\kappa_{M^*,F}$ in Proposition \ref{GTS}. 
\begin{enumerate}
\item For $\tau_1,\tau_2 \in \Omega_{M^*, +}$, we have, $\mathbbm{1}_{K_m^*n_{\tau_1} K_m^*} * \mathbbm{1}_{K_m^*n_{\tau_2} K_m^*} = \mathbbm{1}_{K_m^*n_{\tau_1} n_{\tau_2}K_m^*}$,
\item For $\tau \in \Omega_{M^*, +}$ and for $k_1,k_2 \in K^*$,  we have $\mathbbm{1}_{K_m^*k_1n_\tau k_2 K_m^*} = \mathbbm{1}_{K_m^*k_1 K_m^*}*\mathbbm{1}_{K_m^*n_\tau K_m^*}*\mathbbm{1}_{K_m^*k_2 K_m^*}$. 
\end{enumerate}
\end{proposition}
\begin{proof} Let us calculate $\vol(K_m^* n_\tau K_m^*; dg)$ for $\tau \in \Omega_{M^*, +}$. Recall that $K^* = K_v^*$ where $v$ is a special vertex. 
First, $\vol(K_{v,m}^*n_\tau K_{v,m}^*; dg) = [K_{v,m}^*: K_{v,m}^* \cap n_\tau K_{v,m}^*n_\tau^{-1}]$. 
The Iwahori factorization of $K_{v,m}^*$ gives a set-theoretic bijection 
\begin{align}
K_{v,m}^* \rightarrow \prod_{a \in \Phi^{\red, +}(G^*,A^*)} U_a^*(F)_{v,m} \times T_m^* \times \prod_{a \in \Phi^{\red, -}(G^*,A^*)}U_a^*(F)_{v,m}.
\end{align}
 
Then
\begin{align*}
n_\tau K_{v,m}^* n_\tau^{-1} &\rightarrow \prod_{a \in \Phi^{\red, +}(G^*,A^*)} U_a^*(F)_{v,m+\langle a, \tau \rangle } \times T_m^* \times \prod_{a \in \Phi^{\red, -}(G^*,A^*)}U_a^*(F)_{v,m+\langle a, \tau \rangle}
\end{align*}
This yields a set theoretic bijection
\[K_{v,m}^*/n_\tau K_{v,m}^* n_\tau^{-1} \cap K_{v,m}^* \rightarrow \prod_{a \in \Phi^{\red}(G^*,A^*)  }U_a^*(F)_{ v,m}/ U_a^*(F)_{v,m} \cap n_\tau K_{v,m}^* n_\tau^{-1}.\]
Hence
\begin{align*}
[K_{v,m}^*: K_{v,m}^* \cap n_\tau K_{v,m}^*n_\tau^{-1}] &=  \prod_{a \in \Phi^{\red}(G^*,A^*)} [U_a^*(F)_{ v,m}: U_a^*(F)_{v,m} \cap n_\tau K_{v,m}^* n_\tau^{-1}]\\
& =  \prod_{a \in \Phi^{\red}(G^*,A^*)} [U_a^*(F)_{v,m}:U_a^*(F)_{ v,m} \cap U_a^*(F)_{v,m+\langle a, \tau \rangle }]
\end{align*}
Since $\tau \in \Omega_{M^*, +}$, we have $\langle a, \tau \rangle \geq 0$ for each $a \in \Phi^+(G^*, A^*)$. Hence 
\begin{align*}
[K_{v,m}^*: K_{v,m}^* \cap n_\tau K_{v,m}^*n_\tau^{-1}] & =  \prod_{a \in \Phi^{\red, +}(G^*,A^*)} [U_a^*(F)_{v,m}:U_a^*(F)_{v,m+\langle a, \tau \rangle }]
\end{align*}
Let $a \in \Phi^{\red, +}(G^*,A^*)$. Then 
\begin{align*}
U_a^*(F)_{v,m}/ U_a^*(F)_{v,m+\langle a, \tau\rangle} &\cong U_a^*(\Fu)_{v,m}^{\sigma^*} / ( U_a^*(\Fu)_{v,m+\langle a, \tau \rangle})^{\sigma^*} \\
& \cong \left(U_a^*(\Fu)_{v,m} / U_a^*(\Fu)_{v,m+\langle a, \tau \rangle}\right)^{\sigma^*}.
\end{align*}
The last isomorphism is a consequence of \cite[Lemma 5.1.17]{BT2}. Let $\breve \Phi =\breve \Phi(G^*_\Fu, S^*_\Fu)$. Recall that $G^*_\Fu$ is quasi-split. Let 
\[\Phi^a = \{\bb \in \breve\Phi\;|\; \bb|_{A^*} = a \text{ or } 2a\}.\] 
Then
\[U_a^*(\Fu)_{v,m} = \prod_{\bb \in \Phi^a, \;\bb|_{A^*} = a} U_\bb^*(\Fu)_{v,m} \cdot \prod_{\bb \in \Phi^a \cap \breve\Phi^{\red}, \; \bb|_{A^*} = 2a} U_\bb^*(\Fu)_{v,2m} \]
and
\begin{align*}
U_a^*(\Fu)_{v,m}/U_a^*(\Fu)_{v,m+\langle a, \tau \rangle} = &\prod_{\bb \in \Phi^a, \; \bb|_A = a} \left(U_\bb^*(\Fu)_{v,m}/U_\bb^*(\Fu)_{v,m+\langle a, \tau \rangle}\right) \\&\cdot \prod_{\bb \in \Phi^a \cap \breve\Phi^{\red}, \; \bb|_A = 2a} \left(U_\bb^*(\Fu)_{v,2m}/ U_\bb^*(\Fu)_{v,2m+2\langle a, \tau\rangle}\right) \end{align*}

Let us first deal with the case where $2\bb$ is not a root. Then $U_\bb^* \cong Res_{ \bF_\bb/\Fu} \bbG_a$ where $\bF_\bb$ is the splitting extension of the root $\bb$.

We claim that  $\langle a, \tau \rangle e_\bb \in \bbZ$. To see this, note that with $\Sigma$ denoting the \'echelonnage root system attached to $\Phi(G^*, A^*)$, we have $e_{\bb} a \in \Sigma$ (see  \cite[Section 6.2.23]{BT1} and \cite[Section 4.2.21]{BT2}). Further, by \cite[Section 3.2]{Ros}, we have $\Sigma \subset \bar \Omega_{M^*}^\vee$ and hence  $\langle a, \tau \rangle e_\bb \in \bbZ$.
We have \[U_\bb^*(\bF)_{v,m}/U_\bb^*(\Fu)_{v, m+\langle a,\tau\rangle}  \cong\fO_{\bF_\bb}/ \fp_{\bF_\bb}^{\langle a,\tau \rangle e_\bb}.\] 

For $c \in \Phi(G^*, A^*)$, Let $k_c$ denote the cardinality of the $\sigma^*$-orbit of any root $\bc \in  \breve\Phi(G_\bF^*,S_\bF^*)$ whose restriction to $A^*$ is $c$. 

If $2a$ is not a root, $\Phi^a$ is a single orbit under $\langle \sigma^* \rangle$ and hence,  with $q_a = \# \left(\fO_{\bF_\bb}/ \fp_{\bF_\bb}\right)^{(\sigma^*)^{k_a}}$, we have
\[\#\left(U_a^*(F)_{m}/ U_a^*(F)_{m+\langle a,\tau\rangle} \right)=
q_a^{\langle a,\tau \rangle e_\bb}.\] Similarly, if $2a$ is a root, we have for any $\bb, \bb' \in \Phi^a$ with $\bb|_{A^*} = \bb'|_{A^*} = a$ and with $\bb+\bb'$ a root, 

\[\#\left(U_a^*(F)_{m}/U_a^*(F)_{m+\langle a, \tau\rangle} \right) =
q_a^{\langle a,\tau \rangle e_\bb} \cdot q_{2a}^{\langle 2a,\tau \rangle e_{\bb+\bb'} }. \] 

If $2\bb$ is a root, then there does not exist a reduced root in $\Phi^a$ whose restriction to $A^*$ is $2a$. In this case,  we have 
\[U_a^*(F)_{m}/U_a^*(F)_{m+\langle a,\tau \rangle} \cong \left(U_\bb^*(\bF)_{v,m}/U_\bb^*(\Fu)_{v, m+\langle a,\tau \rangle}\right)^{(\sigma^*)^{k_a}}.\]
Hence
\[\#\left(U_a^*(F)_{m}/U_a^*(F)_{m}\cap U_a^*(F)_{m+\langle a, \tau\rangle} \right) =
q_a^{\langle a,\tau \rangle e_\bb} \cdot q_{2a}^{\langle 2a,\tau \rangle e_{2\bb} }. \] 

Set $e_a = e_\bb$ for any $\bb \in \Phi^a$ with $\bb|_{A^*} = a$. If $2a$ is a root,  set $e_{2a} = e_{\bb+\bb'}$ if $2\bb$ is not a root and let $e_{2a} = e_{2\bb}$ if $2\bb$ is a root. For $a\in \Phi(G^*, A^*)$ note that $e_a$ is the ramification index of the splitting extension of the root $a$ (or the root subgroup $U_a^*$). Then  we have proved that
\begin{align}\label{volcomp}
\vol(K_m^* n_\tau K_m^*; dg) = \prod_{a \in \Phi^{red,+}(G^*,A^*), 2a \notin \Phi(G^*,A^*)} q_{a}^ {\langle a, \tau \rangle e_a} \prod_{a \in \Phi^{red,+}(G^*,A^*), 2a \in \Phi(G^*,A^*)} q_{a}^ {\langle a, \tau\rangle e_a} q_{{2a}}^{\langle 2a, \tau \rangle e_{2a}}
\end{align}

Hence, we have for $\tau_1, \tau_2 \in \Omega_{M^*, +}$, 
\[\vol(K_m^* n_{\tau_1} K_m^*; dg)  \vol(K_m^* n_{\tau_2} K_m^*; dg)  = \vol(K_m^* n_{\tau_1} n_{\tau_2} K_m^*; dg).\]
 Hence parts (a) and (b) are both consequences of  of  \cite[Proposition 2.2]{How85}. 
\end{proof}

\begin{corollary}\label{generators} Choose a finite subset $\Lambda_0 \subset \Omega_{M^*, +}$ such that $\Lambda_0$ contains 0 and generates $\Omega_{M^*, +}$ as a semigroup. Fix a set of representatives $S_{K^*}$ of $K^*/K_m^*$ in $K^*$. The set $\{\mathbbm{1}_{K_m^* n_\tau K_m^*}\;|\; \tau \in \Lambda_0\}\cup \{\mathbbm{1}_{K_m^* k K_m^*}\;|\; k \in S_{K^*}\}$ generates the algebra $\cH(G^*(F), K_m^*)$. 
\end{corollary}
\begin{proof} Recall that the Hecke algebra $\cH(G^*(F),K_m^*)$ is generated as a $\bbC$-vector space by \[\{\mathbbm{1}_{K_m^*k_1 n_\tau k_2^{-1} K_m^*}\;|\;  k_1, k_2 \in S_{K^*}, \tau \in \Omega_{M^*, +}\}.\] Choose $\tau_i \in \Lambda_0$ so that $\tau  = \sum_i \tau_i$. Then $n_\tau  = m \prod_i n_{\tau_i}$ for some $m \in M^*(F)_1 \subset K^*$. Now the corollary follows from the previous proposition. 
\end{proof}

For $\tau \in \Omega_{M^*, +}$, let $G^*_\tau(F) = K^*n_\tau K^*$. This set is a homogeneous space under $K^* \times K^*$ under the action $(k_1,k_2)\cdot g = k_1 g k_2^{-1}$. Let $X$ denote the discrete set of $K_m^*$-double cosets $K_m^* \backslash G^*(F) / K_m^*$ and let $X_\tau \subset X$ denote the set of $K_m^*$- double cosets in $G_\tau^*(F)$. Then $X_\tau$ is a homogeneous space of the finite group $K^*/K_m^* \times K^*/K_m^*$. Let $\Gamma_\tau \subset K^*/K_m^* \times K^*/K_m^*$ be the stabilizer of $K_m^* n_\tau K_m^*$. 

Let $F'$ be another non-archimedean local field that is $e$-close to $F$. Let $\tp':\Omega_{M'^*} \rightarrow M'^*(F'), \tau' \rightarrow n_{\tau'}$ be the section of the Kottwitz homomorphism contructed in Proposition \ref{GTS}. Then, under the isomorphism in Proposition \ref{MmCLF}, we have $n_\tau\mod M_m^* \rightarrow n_{\tau'} \mod M_m'^*$. 

By Proposition \ref{CWG}, we know that $W(G^*,A^*) \backslash \Omega_{M^*} \cong W(G'^*,A'^*) \backslash \Omega_{M'^*}$. Recall that $\Omega_{M^*, +}$ is the set of dominant elements of $\Omega_{M^*}$. Note that $\Omega_{M^*, +}$ contains 0. Under the isomorphism $\Omega_{M^*} \cong \Omega_{M'^*}$, $\Omega_{M^*,+}$ maps to $\Omega_{M'^*, +}$. Consider the isomorphism $p_{m}^*: K^*/K_m^* \times K^*/K_m^* \rightarrow K'^*/K_m'^* \times K'^*/K_m'^*$ induced by Theorem \ref{PGCLFI}. Then for each $\tau \in \Omega_{M^*, +}$, it is clear that $p_{m}^*(\Gamma_\tau) = \Gamma_{\tau'}$. This allows us to construct a bijection $X \rightarrow X'$ and hence an isomorphism of $\bbC$-linear spaces
\[h_{m}^*: \cH(G^*(F), K_m^*) \rightarrow \cH(G'^*(F'), K_m'^*).\]
Our goal is to prove that there exists an $l\geq e$ such that for any $F'$ that is $l$-close to $F$, the above map is an algebra isomorphism.

\begin{lemma}\label{nC} Let $\Lambda \subset \Omega_{M^*, +}$ be a finite subset and let $G_\Lambda^*(F) = \cup_{\tau \in \Lambda} G_\tau^*(F)$. 
\begin{enumerate}
\item There exists a natural number $N = N_\Lambda \geq m$ such that for all $g \in G_\Lambda^*(F)$, $gK_N^*g^{-1} \subset K_m^*$.
\item Choose $l\geq \max(N, e)$ large enough so that for any $F'$ that is $l$-close to $F$, Theorem \ref{PGCLFI} yields an isomorphism $p_N^*:K^*/K_N^* \rightarrow K'^*/K_N'^*$. Then for each $h_1,h_2 \in \cH(G^*(F), K_m^*)$ supported on $G_\Lambda^*(F)$, we have
\[ h_{m}^*(h_1 *h_2) = h_{m}^*(h_1) * h_{m}^*(h_2).\]
\end{enumerate}
\end{lemma}
\begin{proof} We will prove that for each $\tau \in \Lambda$, there exists $N_\tau \geq m$ such that  $n_\tau K_{N_\tau}^*n_\tau^{-1} \subset K_m^*$, since then, (a) would hold for $N_\Lambda = \max\{N_\tau\;|\; \tau \in \Lambda\}$. Now, fix $\tau \in \Lambda$. Let $N_\tau$ be large enough so that $N_\tau +\langle a, \tau \rangle \geq m$ for each $a \in \Phi(G^*, A^*)$. Now, $K_{ N_\tau}^* = \langle T_{N_\tau}^*, U_a^*(F)_{v, N_\tau}\;|\; a \in \Phi(G^*,A^*) \rangle$. Then $n_\tau K_{N_\tau}^* n_\tau^{-1}= \langle T_{N_\tau}^*, n_\tau U_a(F)_{v, N_\tau}^*n_\tau^{-1}\; |\; a \in \Phi(G^*, A^*) \rangle = \langle T_{N_\tau}^*,  U_a(F)^*_{v, N_\tau +\langle a, \tau \rangle}\;|\; a \in \Phi(G^*, A^*) \rangle$. By the choice of $N_\tau$, it follows that $n_\tau K_{N_\tau}^* n_\tau^{-1}\subset K_m^*$. This finishes the proof of (a).

 Write $h_i= \mathbbm{1}_{K_m^*g_i K_m^*}, i=1,2$ for $g_i \in G_\Lambda^*(F)$. Note that $h_1*h_2(x) = \vol(K_m^*g_1K_m^* \cap K_m^* g_2K_m^* x, dg)$. Then (a) implies that $K_m^*g_1K_m^* \cap K_m^* g_2K_m^* x$ is $K_{N}^*$-bi-invariant. 
 Now,
 \[h_1*h_2 = \sum_{\tau \in \Omega_{M^*, +}} \sum_{K_m k_1 n_\tau k_2^{-1} K_m \in X_\tau} \vol(K_m^*g_1K_m^* \cap K_m^* g_2K_m^* k_1n_\tau k_2^{-1}, dg) \mathbbm{1}_{K_N^* k_1n_\tau k_2^{-1}K_N^*}\]
 The calculation in Proposition \ref{simpleconv} implies that if $F$ and $F'$ are $l$-close, then $$\vol(K_m^*g_1K_m^* \cap K_m^* g_2K_m^* k_1n_\tau k_2^{-1}, dg) = \vol(K_m'^*g_1'K_m^* \cap K_m'^* g_2'K_m^* k_1'n_\tau' k_2'^{-1}, dg')$$ where $p_{N}^*(k_1 \mod K_N^*, k_2 \mod K_N^*) = (k_1'\mod K_N'^*, k_2'\mod K_N'^*)$. Further, \[h_{N}^*(\mathbbm{1}_{K_N^* k_1n_\tau k_2^{-1}K_N^*}) = \mathbbm{1}_{K_N'^* k_1'n_\tau' k_2'^{-1}K_N'^*}.\] This implies that $h_{N}^*(h_1 * h_2) = h_{N}^*(h_1) * h_{N}^*(h_2)$. Since $h_{N}^*$ agrees with $h_{m}^*$ on $\cH(G^*(F), K_m^*)$, this finishes the proof of (b). 
 \end{proof}

\subsection{Proof of Theorem \ref{MainTheorem}} With the above ingredients in place, the proof of Theorem \ref{MainTheorem} is identical to that of Kazhdan \cite{kaz86}. We write it down for completeness.
We know by \cite[Theorem 2.13 and Corollary 3.4]{Ber84} that the Hecke algebra $\cH(G^*(F),K_m^*)$ is finitely presented. Let $x_1, x_2 \cdots x_p$ be a finite set of generators and let $R_1, R_2, \cdots R_q$ be a finite set of relations among these generators, that is, these are non-commutative polynomials in $p$ variables such that $R_i(x_1 \cdots x_p) = 0$ for $1 \leq i \leq q$. 

We index the elements $\{\mathbbm{1}_{K_m^*gK_m^*}\;|\; g \in \Lambda_0 \cup S_{K^*}\}$ as $f_1, f_2 \cdots f_r$, and by Corollary \ref{generators}, the elements $f_1, f_2, \cdots f_r$ form a system of generators for $\cH(G^*(F), K_m^*)$. Let $G_i, 1 \leq i \leq p$ be polynomials in $r$ variables such that $G_i(f_1, f_2 \cdots f_r) = x_i,\; 1 \leq i \leq p$. Similarly, let   $F_i$ be polynomials in $p$ variables such that $F_i(x_1, x_2 \cdots x_p) = f_i,\; 1 \leq i \leq r$. Let $N_0$ be the maximal degree of the polynomials $R_i(G_1, G_2, \cdots G_p), 1\leq i \leq q$, and $F_i(G_1, G_2 \cdots, G_p), 1 \leq i \leq r$. Let $\Lambda \subset \Omega_{M^*, +}$ be a finite subset such that all possible products of $N_0$ terms of the $f_i$'s is contained in $G^*_\Lambda(F)$.  Choose $l$ as in Lemma \ref{nC}(b). Suppose $F$ and $F'$ are $l$-close. Let $A:= \bbC (x_1, \cdots x_p 
)$. We have an algebra homomorphism $\fe: A \rightarrow \cH(G'^*(F), K_m'^*), x_i \rightarrow G_i(f_1', f_2' \cdots f_r')$, where $f_i' = h_m^*(f_i)$.  It follows from Lemma \ref{nC}(b) that 
\[\fe(R_j(x_1, x_2, \cdots x_p))=0\; \forall j=1, 2 \cdots q.\]
Hence we obtain an algebra homomorphism $$\bar \fe: \cH(G^*(F), K_m^*) \rightarrow \cH(G'^*(F'),K_m'^*).$$  By Lemma \ref{nC}(b) again, we have $\bar\fe(f_i) = h_{m}^*(f_i), \; 1 \leq i\leq r$. Recall that $X$ is the set of discrete double cosets $K_m^*\backslash G^*(F)/K_m^*$ and the characteristic functions of elements of $X$ gives a $\bbC$-basis of $\cH(G^*(F),K_m^*)$. By Proposition \ref{simpleconv}, $\bar\fe = h_{m}^*$ on these characteristic functions. Hence $h_{m}^*$ is an algebra isomorphism.

\bibliographystyle{amsalpha}

\end{document}